\documentclass[12pt]{article}

\usepackage{graphicx}
\usepackage{amsthm,amsmath,amssymb}
\usepackage[noadjust,sort]{cite}

\date{}

\normalsize

\numberwithin{equation}{section}

\def\nfrac#1#2{{\textstyle\frac{#1}{#2}}}
\def\({\bigl(}
\def\){\bigr)}

\newtheorem{thm}{Theorem}[section]
\newtheorem{cor}{Corollary}[section]
\newtheorem{lemma}{Lemma}[section]
\newtheorem{conj}{Conjecture}

\def\dfrac#1#2{\lower0.15ex\hbox{\large$\frac{#1}{#2}$}}
\def\({\bigl(}
\def\){\bigr)}
\def\sumprime{\mbox{$\vphantom{\displaystyle\sum}$}'}

\def\sumppd{\mathop{\sum\nolimits'}\limits}
\def\dmax{d_{\mathrm{max}}}

\def\A{\mathcal{A}}
\def\M{\mathcal{M}}

\let\eps=\varepsilon
\def\Deltait{\mathit{\Delta}}
\def\Lambdait{\mathit{\Lambda}}

\def\dvec{\boldsymbol{d}}
\def\mvec{\boldsymbol{m}}
\def\pvec{\boldsymbol{p}}
\def\zvec{\boldsymbol{z}}
\def\betavec{\boldsymbol{\beta}}

\def\OO{\widetilde{O}}

\def\E{\operatorname{\mathbb{E}}}
\def\Var{\operatorname{Var}}
\def\Prob{\operatorname{Prob}}
\def\Reals{{\mathbb{R}}}
\def\abs#1{\lvert#1\rvert} \let\card=\abs

\def\Bin{\operatorname{Bin}}
\def\PB{\operatorname{PB}}
\def\nicebreak{\vskip 0pt plus 50pt\penalty-300\vskip 0pt plus -50pt }

\begin{document}

\title{Counting loopy graphs with given degrees}

\author{
Catherine~Greenhill\\
\small School of Mathematics and Statistics\\[-0.9ex]
\small University of New South Wales\\[-0.9ex]
\small Sydney, Australia 2052\\[-0.3ex]
\small\texttt{csg@unsw.edu.au}
\and
Brendan~D.~McKay\vrule width0pt height2ex\thanks
 {Research supported by the Australian Research Council.}\\
\small Research School of Computer Science\\[-0.9ex]
\small Australian National University\\[-0.9ex]
\small Canberra ACT 0200, Australia\\[-0.3ex]
\small\texttt{bdm@cs.anu.edu.au}
}

\maketitle

\begin{abstract}
Let $\dvec=(d_1,d_2,\ldots, d_n)$ be a vector of nonnegative
integers.  We study the number of symmetric 0-1 matrices
whose row sum vector equals~$\dvec$.
While previous work has focussed on the case of zero diagonal, 
we allow diagonal entries to equal~1.    Specifically, for $D\in\{ 1,2\}$
we define the set $\mathcal{G}_D(\dvec)$ of all $n\times n$ symmetric
0-1 matrices with row sums given by $\dvec$, where each diagonal entry is
multiplied by $D$ when forming the row sum.
We obtain asymptotically precise formulae for $|\mathcal{G}_D(\dvec)|$
in the sparse range (where, roughly, the maximum
row sum is $o(n^{1/2})$), and in the dense range (where, roughly,
the average row sum is proportional to~$n$ and the row
sums do not vary greatly).
The case $D=1$ corresponds to enumeration by the usual row sum of matrices.
The case $D=2$ corresponds to enumeration by degree sequence
of undirected graphs with loops but no repeated edges, due to the convention
that a loop contributes 2 to the degree of its incident vertex.
We also analyse the distribution of the trace of a random element of 
$\mathcal{G}_D(\dvec)$, and prove that it is well approximated by a 
binomial distribution in the dense
range, and by a Poisson binomial distribution in the sparse range.
\end{abstract}

\section{Introduction}\label{s:intro}

Let $\dvec=(d_1,d_2,\ldots,d_n)$ be a vector of nonnegative
integers. Define
$G(\dvec)$ to be the number of $n\times n$ symmetric matrices
over $\{ 0,1\}$ with zero diagonal, such that
row $j$ sums to $d_j$, for $j=1,\ldots,n$.

The quantity $G(\dvec)$ has been well studied, as cited
below.  In this paper we consider the case where the
diagonal need not be zero. 
For $D\in\{ 1,2\}$
define $\mathcal{G}_D(\dvec)$ to be the set of $n\times n$ symmetric
matrices $A=(a_{jk})$ over $\{ 0,1\}$ such that
\[  D a_{jj} + 
   \!\!\sum_{1\leq k\leq n,\, k\ne j} \!\! a_{jk}  = d_j \quad 
                      \text{ for }\, j = 1,\ldots, n.
\]
We wish to find an asymptotic formula for
\[ G_D(\dvec) = \card{\mathcal{G}_D(\dvec)}. \]
The case of $D=1$ corresponds to enumeration by row sum of symmetric
0-1 matrices.
If we interpret $A$ as the adjacency matrix of a simple undirected
graph with loops, then the case of $D=2$ corresponds to enumeration
by degree sequence of simple undirected graphs with loops. 
Such graphs arise in various applications including the
study of graph homomorphisms~\cite{GHRY} and sign patterns~\cite{CHMM}.

Throughout the paper we will refer to a nonzero entry on the
diagonal of a 0-1 matrix as a \emph{loop}.
For $\ell  = 0, 1,\ldots, n$, let
$\mathcal{G}_D(\dvec,\ell)$ be the set of matrices in $\mathcal{G}_D(\dvec)$
with exactly $\ell$ loops (that is, with trace $\ell$), 
and let $G_D(\dvec,\ell)= \card{\mathcal{G}_D(\dvec,\ell)}$.  
Clearly we have $G_D(\dvec,0)=G(\dvec)$  and 
$G_D(\dvec) = \sum_{\ell=0}^n G_D(\dvec,\ell)$.
We also note here that $G_1(\dvec,\ell)=0$ unless
$\sum_{j=1}^n d_j$ has the same parity as $\ell$,
and $G_2(\dvec,\ell)=0$ unless $\sum_{j=1}^n d_j$ is even.

When $d_j=d$ for $j=1,\ldots, n$, we write $G_D(\dvec)=G_D(n,d)$
and refer to this as the regular case.

We will use the following parameters frequently:
\def\fixwid#1{\hbox to 5em{$\displaystyle#1$\hss}}
\begin{align*}
  S &= \fixwid{ \sum_{j=1}^n d_j,} &
  d &= \fixwid{\frac{S}{n},} \\
  \lambda &= \fixwid{\frac{d}{n-1},} &
  \dmax &=\max_j \, d_j, \\
   R &= \sum_{j=1}^n \, (d_j - d)^2, &
   S_r &= \sum_{j=1}^n\, [d_j]_r\quad(r=2,3),
\end{align*}
where $[a]_r = a(a-1)\cdots (a-r+1)$ denotes the falling factorial.

Throughout the paper, the asymptotic notation $O(f(m))$ refers
to the passage of the variable~$m$ to infinity. (Usually $m=n$ or $m=S$.)
In the dense setting we also use a modified notation $\OO(f(n))$,
which is to be taken as a shorthand for $O\(f(n)n^{c\eps}\)$
with $c$ a numerical constant (perhaps a different constant for each
occurrence).
We write $\Omega(g(n))$ to indicate any function which
is greater than $C g(n)$ for some constant $C>0$ and sufficiently
large~$n$.

\medskip
It appears that there is very little prior research on $G_D(\dvec)$.
The most general result, by Bender and Canfield, dates from 1978.

\begin{thm}[\cite{BC}]
\label{BC78}
Suppose that $1\leq \dmax = O(1)$. 
Then 
\begin{align*}
G_1(\dvec) &= \frac{1}{\sqrt{2}}\,
  \biggl(\frac{S}{e}\biggr)^{\!S/2} \biggl(\,\prod_{j=1}^n d_j!\biggr)^{\!-1}
   \exp\biggl(\sqrt{S} - \frac{1}{4} - \frac{S_2}{S}
   - \frac{S_2^2}{4S^2}  + o(1) \biggr)
\end{align*}
uniformly as $S\to\infty$.
\end{thm}

Note that $G_1(n,1)$ is the number of involutions on $n$ letters (and also 
the number of Young tableaux with $n$ cells, see~\cite[A000085]{sloane}).
The asymptotic expansion of $G_1(n,1)$ was previously known, 
see~\cite{chowla, moser}.  We found no prior asymptotic
work on $G_2(\dvec)$ at all.

In the case of $D=1$, a graph with $n$ vertices and $\ell$ loops can be
mapped to a graph with $n+1$ vertices and no loops, by
introducing a new vertex and replacing each loop by an edge to
this vertex.   This mapping is bijective and hence
\[
    G_1\((d_1,\ldots,d_n),\ell\) = G\((d_1,\ldots,d_n,\ell)\).
\]
However, this doesn't seem to be of much use in asymptotic enumeration,
since the important values of $\ell$ place the degree sequence
$(d_1,\ldots,d_n,\ell)$ out of range of existing explicit estimates.

\medskip

Our approach to estimating $G_D(\dvec)$ will be to sum over all possible
diagonals using the existing estimates for $G(\dvec)$.
The main estimates we will use are the following two theorems.
The history of previous results on $G(\dvec)$ is summarized
in~\cite{BDM85} and~\cite{MWreg}.

McKay and Wormald~\cite[Theorem 5.2]{MWsparse} proved the following asymptotic
formula for $G(\dvec)$ in the sparse regime.

\begin{thm}[\cite{MWsparse}]\label{sparsenum}
If\/ $1\leq \dmax = o(S^{1/3})$ 
then
\[
  G(\dvec) = \frac{S!}{(S/2)!\,2^{S/2}\prod_{j=1}^n d_j!}\,
    \exp\biggl( -\frac{S_2}{2S}
     - \frac{S_2^2}{4S^2} - \frac{S_2^2S_3}{2S^4} + \frac{S_2^4}{4S^5} +
                \frac{S_3^2}{6S^3}
       + O\biggl(\frac{\dmax^3}{S}\biggr)
   \biggr),
\]
uniformly as $S\to\infty$, with $S$ even. 
\end{thm}

In the case of dense matrices, the following result was due
to McKay and Wormald~\cite{MWreg} except that we will use an
improved error term from a generalization by McKay~\cite{ranx}.
A less explicit formula allowing a wider variation of the degrees
was proved by Barvinok and Hartigan~\cite{BarvHart}.

\begin{thm}[\cite{ranx}]\label{densenum}
Let $a,b>0$ be constants such that $a+b<\frac12$.  
Then there is a constant $\varepsilon_0 = \varepsilon_0(a,b) > 0$ such that
the following holds.
Suppose that $d_j-d$ is uniformly $O(n^{1/2+\eps_0})$ for $j=1,\ldots, n$
and that 
\[
     \min\{d,n-d-1\} \ge \frac {n}{3a\log n}
\]
for sufficiently large $n$. 
Then provided $S$ is even we have
\begin{equation}\label{Gdense}
G(\dvec) =\sqrt 2\,\(\lambda^\lambda (1-\lambda)^{1-\lambda}\)^{\binom n2}
   \exp\biggl( \,\frac14 - \frac{R^2}{4\lambda^2(1-\lambda)^2n^4}
                + O(n^{-b}) \biggr)
     \prod_{j=1}^n \binom{n{-}1}{d_j}.
\end{equation}
\end{thm}

This formula also matches the sparse case under slightly
more restricted conditions than Theorem~\ref{sparsenum} and is conjectured
to hold in the intermediate domain as well 
(see~\cite[Theorem 2.5]{MWdegseq} and the conjecture stated 
immediately thereafter).

Note that Theorem~\ref{densenum} remains true if $\varepsilon_0(a,b)$ is
decreased (but is still positive), since the conditions of the theorem
become stronger.

We now state our main enumeration theorems, starting with
the dense regime.

\begin{thm}\label{densetheorem}
Let $a,b>0$ be constants such that $a+b<\frac12$.  
Then there is a constant $\varepsilon = \varepsilon(a,b) > 0$ such that
the following holds.
Suppose that $d_j-d$ is uniformly $O(n^{1/2+\eps})$ for $j=1,\ldots, n$
and that 
\begin{equation}
\label{dcondition}
     \min\{d,n-d\} \ge \frac {n}{3a\log n}
\end{equation}
for sufficiently large $n$.
For $D\in \{ 1,2\}$, define
    \[\mu_D = \frac{d}{n+D-1},\] 
    and let
\begin{align*}
  Q_1(\dvec,\ell) &= 
     \frac{1}{4} + \frac{(\ell- d)^2}{4d(n-d)}
       - \frac{(\ell-d)^2R}{2d^2(n-d)^2}
       - \frac{R^2}{4d^2(n-d)^2}, \\[0.5ex]
  Q_2(\dvec,\ell) &= \frac{1}{4} - \frac{\ell(n-\ell)}{\mu_2(1-\mu_2)n^2}
     - \frac{R^2}{4\mu_2^2(1-\mu_2)^2n^4}
     + \frac{R}{\mu_2(1-\mu_2)n^2} \\
     &{\kern12mm} + \frac{(1-2\mu_2)(\ell-\mu_2 n)R}{\mu_2^2(1-\mu_2)^2n^3}
     - \frac{2(\ell-\mu_2 n)^2R}{\mu_2^2(1-\mu_2)^2 n^4}. 
     \end{align*}
When $\ell$ has the same parity as $S$ we have
\begin{align*}
   G_1(\dvec,\ell) = 
  \sqrt 2\,\(\mu_1^{\mu_1} &(1-\mu_1)^{1-\mu_1}\)^{ n^2/2}
     \binom{n}{\ell}
     \mu_1^{\ell/2}\, (1-\mu_1)^{(n-\ell)/2}\\[-1ex]
   &{\qquad}\times\exp\( Q_1(\dvec,\ell) + O(n^{-b}) \)
     \prod_{j=1}^n \binom{n}{d_j},
\end{align*}
while for\/ $\ell=0,\ldots, n$ and even $S$ we have
\begin{align*}
   G_2(\dvec,\ell) = 
  \sqrt 2\,\(\mu_2^{\mu_2} &(1-\mu_2)^{1-\mu_2}\)^{\binom{n+1}{2}}
     \binom{n}{\ell}
   \mu_2^\ell\, (1-\mu_2)^{n-\ell}\\[-1ex] 
   &{\quad}\times\exp\( Q_2(\dvec,\ell) + O(n^{-b}) \)
     \prod_{j=1}^n \binom{n+1}{d_j}.
\end{align*}
Defining
\[ 
  \bar\ell_1 = \frac{d^{1/2} n}{d^{1/2}+(n-d)^{1/2}},\qquad
  \bar\ell_2 = \mu_2 n = \frac{dn}{n+1},
\]
we have
\begin{align*}
   G_1(\dvec) &= 
  \frac{1}{\sqrt{2}}
	\(\mu_1^{\mu_1} (1-\mu_1)^{1-\mu_1}\)^{n^2/2}
     \(\mu_1^{1/2}+(1-\mu_1)^{1/2}\)^{\!n}\\[-1ex]
   &{\kern40mm}\times \exp\( Q_1(\dvec,\bar\ell_1) + O(n^{-b}) \)
     \prod_{j=1}^n \binom{n}{d_j}
\end{align*}
and, for even $S$,
\begin{align*}
   G_2(\dvec) &= 
   \sqrt{2}\,
	\(\mu_2^{\mu_2} (1-\mu_2)^{1-\mu_2}\)^{\binom{n+1}{2}}
  \exp\( Q_2(\dvec,\bar\ell_2) + O(n^{-b}) \)
     \prod_{j=1}^n \binom{n{+}1}{d_j}.  \quad \qed
\end{align*}
\end{thm}

In Theorem~\ref{distribution} we will prove that
$\bar{\ell}_D$ is close to the expected number of loops in a randomly
chosen element of $\mathcal{G}_D(\dvec)$.
For the reader's convenience, we note that
\begin{align*}
  Q_2(\dvec,\bar\ell_2) &= -\frac{1}{4}
    \biggl( 1-\frac{R}{\mu_2(1-\mu_2)n^2}\biggr)
    \biggl( 3-\frac{R}{\mu_2(1-\mu_2)n^2}\biggr).
\end{align*}
Unfortunately, the expression for $Q_1(\dvec,\bar\ell_1)$ does
not simplify much.
In the case of regular graphs we have $R=0$, so the formulae
for $Q_D(\dvec,\ell)$ simplify greatly and in particular
\[
    Q_1(\dvec,\bar\ell_1) = \frac{n}{2n+4\sqrt{d(n-d)}}\, .
\]

Our main result for the sparse case is the following.

\begin{thm}
Suppose that $1\leq \dmax = o(S^{1/3})$. 
Then 
\begin{align*}
G_1(\dvec) &= \frac{1}{\sqrt{2}}\,
  \biggl(\frac{S}{e}\biggr)^{\!S/2}\, \biggl(\,\prod_{j=1}^n d_j!\biggr)^{\!-1}
   \exp\biggl(\sqrt{S} - \frac{1}{4}
   - \frac{S_2}{S} - \frac{S_2^2}{4S^2} \\
 & \hspace*{1cm} {} + \frac{7}{24S^{1/2}}
  + \frac{S_2}{S^{3/2}} + \frac{S_3}{3S^{3/2}}
  + \frac{S_2^2}{2S^{5/2}} - \frac{S_2^2S_3}{2S^4} + 
   \frac{S_2^4}{4S^5} + \frac{S_3^2}{6S^3} + O\biggl(\frac{\dmax^3}{S}\biggr)
   \biggr)
\end{align*}
uniformly as $S\to\infty$, and
\begin{align*}
G_2(\dvec) &= \sqrt{2} \, \biggl(\frac{S}{e}\biggr)^{\!S/2}
   \, \biggl(\,\prod_{j=1}^n d_j!\biggr)^{\!-1}
   \exp\biggl(\frac{S_2}{2S} - \frac{S_2^2}{4S^2} - \frac{S_2^2S_3}{2S^4}
       +\frac{S_2^4}{4S^5} + \frac{S_3^2}{6S^3} + O\biggl(\frac{\dmax^3}{S}\biggr)
   \biggr)
\end{align*}
uniformly as $S\to\infty$ with $S$ even.
\label{sparsetheorem}
\end{thm}
If $S$ is even then we may replace
the factor $\sqrt{2}\,(S/e)^{S/2}$ by $S!/\((S/2)!\,2^{S/2}\)$.
In the regular case the formulae simplify as follows.

\begin{cor}
\label{sparseregular}
Suppose that $1\leq d = o(n^{1/2})$.  Then
\begin{align*}
G_1(n,d) &= 
   \frac{1}{\sqrt{2}}\, (d!)^{-n}\, \biggl(\frac{nd}{e}\biggr)^{\!nd/2}\,\\
           & \hspace*{1cm} {} \times
          \exp\biggl(\frac{2-2d-d^2}{4} 
	           + \frac{24(n-1)d + 20d^2 + 11}{24\sqrt{nd}} 
		             - \frac{d^3}{12n} + O\biggl(\frac{d^2}{n}\biggr)
			      \biggr)
\end{align*}
uniformly as $n\to\infty$, and
\begin{align*}
G_2(n,d) &= \sqrt{2}\, (d!)^{-n} \biggl(\frac{nd}{e}\biggr)^{\!nd/2}
                 \exp\biggl(-\frac{(d-1)(d-3)}{4} - \frac{d^3}{12n} 
		            + O\biggl(\frac{d^2}{n}\biggr)\biggr)
\end{align*}
uniformly as $n\to\infty$ with $nd$ even.
\end{cor}
\noindent Again, if $nd$ is even then the factor 
$\sqrt{2}\,(nd/e)^{nd/2}$
may be replaced by 
$(nd)!/\((nd/2)!\, 2^{nd/2}\)$.

Theorems~\ref{densetheorem} and~\ref{sparsetheorem} are proved
in Section~\ref{s:dense} and~\ref{s:sparse}, respectively.
Along the way we prove some technical results 
(Lemmas~\ref{Ul1},~\ref{technical},~\ref{parity}) which 
may be of independent interest.
But first, in Section~\ref{ss:distribution} we state a theorem on the
distribution of the trace of a random element of
$G_D(\dvec)$, and discuss some interesting features of this
distribution.  Theorem~\ref{distribution} is proved
in Section~\ref{s:distributionproof}.
Finally in Section~\ref{s:conjecture} we state a conjecture
regarding the number of regular graphs with loops, for all possible
degrees.

\subsection{The distribution of the trace}\label{ss:distribution}

The calculations we will give in the process of proving
Theorems~~\ref{densetheorem} and~\ref{sparsetheorem} will provide
some information on the distribution of the trace of a
random element of~$\mathcal{G}_D(\dvec)$.  We summarize that
information here.

For $\pvec=(p_1,\ldots,p_n)\in [0,1]^n$, let $X_1,\ldots,X_n$ be
independent random variables with $\Prob(X_j=0)=1-p_j$ and
$\Prob(X_j=1)=p_j$ for each~$j$.
The \textit{Poisson binomial distribution} $\PB(\pvec)$ 
is the distribution of $\sum_{j=1}^n X_j$.
Define
\[  \PB(\pvec,\ell) = \Prob\Bigl(\,\sum_{j=1}^n X_j = \ell\Bigr). \]
The special case $\pvec=(p,p,\ldots,p)$ gives the familiar binomial
distribution,
\[  \PB((p,\ldots,p),\ell) =\Bin(n,p,\ell) = \binom{n}{\ell} p^\ell(1-p)^{n-\ell}. \]

\begin{thm}
\label{distribution}
Let $Y_D=Y_D(\dvec)$ be the random variable given by the trace of
an element of $\mathcal{G}_D(\dvec)$ chosen uniformly at random.
\begin{enumerate}
\item[(i)] If the conditions of Theorem~\ref{densetheorem} hold then,
 for\/ $\ell=0,\ldots,n$,
 \begin{align*}
   \Prob(Y_1=\ell) &= \(2 + O(n^{-b})\) \Bin(n,\bar\ell_1/n,\ell)
        +O(e^{-n^{\Omega(1)}}),\\
   \E(Y_1) &= \bar\ell_1\(1+O(n^{-b})\),\\
   \Var(Y_1) &= \bar\ell_1(1-\bar\ell_1/n)\(1+O(n^{-b})\),
                \displaybreak[0]\\[1ex]
   \Prob(Y_2=\ell) &= \(1 + O(n^{-b})\) \Bin(n,\bar\ell_2/n,\ell)
       + O(e^{-n^{\Omega(1)}}),\\
   \E(Y_2) &= \bar\ell_2\(1+O(n^{-b})\),\\
   \Var(Y_2) &= \bar\ell_2(1-\bar\ell_2/n)\(1+O(n^{-b})\),
\end{align*}
where $\ell$ must have the same parity as $S$ in the $D=1$ case
and $S$ must be even in the $D=2$ case.
\item[(ii)] Define
$\pvec'=(p'_1,\ldots,p'_n)$ and $\pvec''=(p''_1,\ldots,p''_n)$, where
for $j=1,\ldots,n$,
\begin{align*}
   p'_j &= \frac{d_j}{\sqrt S} - \frac{d_j(2d_j-1)}{2S}
     + \frac{d_j^3}{S^{3/2}}
     + \frac{d_j(d_j-2)S_2}{S^{5/2}}
     - \frac{d_jS_2^2}{2S^{7/2}}, \\[0.6ex]
   p''_j &= \frac{d_j(d_j-1)}{S}.
\end{align*}
If the conditions of Theorem~\ref{sparsetheorem} hold then,
 for\/ $\ell=0,\ldots,n$,
  \begin{align*}
   \Prob(Y_1=\ell) 
     &= \biggl(2 + O\biggl(\frac{\dmax^3}{S}+S^{-1/3}\biggr)\biggr)
                    \PB(\pvec',\ell) +O(e^{-S^{\Omega(1)}}),\\
    \E(Y_1) &= \sqrt{S} - \frac{S_2}{S}  - \frac{1}{2} 
        +O\biggl(\frac{\dmax^3}{S^{1/2}}\biggr) \\
     \Var(Y_1) &= \sqrt{S} - \frac{2S_2}{S} -1
        +O\biggl(\frac{\dmax^3}{S^{1/2}}\biggr) \displaybreak[0] \\
   \Prob(Y_2=\ell) &= \biggl(1 + O\biggl(\frac{\dmax^2}{S^{2/3}}+S^{-1/3}\biggr)\biggr)
                                 \PB(\pvec'',\ell) + O(e^{-S^{\Omega(1)}}),\\
  \E(Y_2) &=  \biggl(1 + O\biggl(\frac{\dmax^3}{S}\biggr)\biggr)\,
         \frac{S_2}{S},\\
  \Var(Y_2) &=  \biggl(1 + O\biggl(\frac{\dmax^3}{S}\biggr)\biggr)\,
         \frac{S_2}{S},
\end{align*}
where $\ell$ must have the same parity as $S$ in the $D=1$ case
and $S$ must be even in the $D=2$ case.
\end{enumerate}

\end{thm}

The parameter $\mu_D$ can be thought of as measuring the
density of entries equal to~1, while $Y_D/n$ is the density
of loops in a randomly chosen element of $\mathcal{G}_D(\dvec)$.
In the dense range of Theorem~\ref{distribution} we see that $Y_2/n$
is concentrated near the same value $\mu_2$, while $Y_1/n$ is
concentrated near 
\[ \frac{\bar{\ell}_1}{n} = 
       \frac{\sqrt{\mu_1}}{\sqrt{\mu_1} + \sqrt{1-\mu_1}}.\]
Figure~\ref{densities} illustrates this curious difference between
$D=1$ and~$D=2$.

\begin{figure}[ht]
\unitlength=1cm
\begin{picture}(16,7.5)(0,0.5)
\put(4,0.3){\includegraphics[scale=0.4]{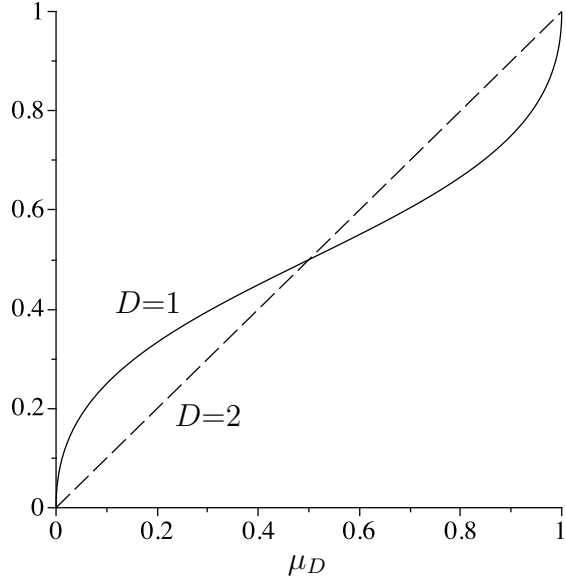}}
\put(6.5,2){$D{=}2$}
\put(5.7,3.5){$D{=}1$}
\put(8,0.1){$\mu_D$}
\end{picture}
\label{densities}
\caption{The expected density of the diagonal as a function of the overall
density $\mu_D$.}
\end{figure}

When $D=1$, Theorem~\ref{distribution} tells
us that the most significant term in $\E(Y_D)$ depends
only on $S$ and not on $\dvec$, within the range of
$\dvec$ values allowed by the theorem.
To explore this further, let $\A_n(S)$ be 
the set of all $n\times n$
symmetric 0-1 matrices with exactly $S$ entries equal to 1.
The number of matrices in $\A_n(S)$ with exactly
$\ell$ loops~is
\[ \binom{n}{\ell}\binom{\binom{n}{2}}{\frac{S-\ell}{2}}\]
when $S$ and $\ell$ have the same parity, and 0 otherwise.
For $1\le S\le n^2-1$, it can be proved that the maximum value of
this function occurs either at $\bar\ell_1$ rounded up to an integer
of the same parity as~$S$ or $\bar\ell_1$ rounded down to such an
integer.  On the basis of experiments, we conjecture that the mean
number of loops in $\A_n(S)$ always lies in $(\bar\ell_1-\frac12,
\bar\ell_1+\frac12)$.

Also note that $\bar\ell_1\sim \sqrt{S}$ for $S=o(n^2)$,
matching the leading term of $\E(Y_1)$ in the sparse case.

When $D=2$ we consider instead the set $\mathcal{B}_n(S)$ of
all graphs with loops allowed, with $n$ vertices and $S/2$ edges
(loops counting twice).
Matrices which correspond to graphs in $\mathcal{B}_n(S)$
can be formed by choosing $S/2$ entries on or below the main diagonal,
setting these equal to 1, then adding this matrix to its transpose.
(Nonzero entries on the diagonal all equal 2, which is their 
contribution to the row sum.)
The number of graphs in $\mathcal{B}_n(S)$ with exactly $\ell$
loops~is
\[ \binom{n}{\ell}\binom{\binom{n}{2}}{S/2 - \ell}.\]
Up to scaling, this is the hypergeometric distribution with parameters
$\(\binom{n+1}{2}, n, S/2\)$ and mean $S/(n+1)=\mu_2 n$.

The binomial distributions in part (i) of the theorem are asymptotically
normal, as is well known.  The Poisson binomial distributions in part (ii)
of the theorem are asymptotically normal for $Y_1$ (see~\cite{DPR}), and
asymptotically Poisson for $Y_2$, by Le Cam's Theorem~\cite{lecam} 
(see also~\cite[Equation 1.1]{BH}).

\nicebreak
\section{The dense case}\label{s:dense}

In this section we prove Theorem~\ref{densetheorem}.

\subsection{A technical lemma}\label{ss:dense-technical}

We will require a technical lemma which might be of some
independent interest.
If $\betavec=(\beta_1,\ldots,\beta_n)$ is a vector of real numbers
and $\ell=0,\ldots,n$, define
\[ U_\ell(\betavec) = \sum_{1\le j_1<\cdots<j_\ell\le n}
   \,\prod_{s=1}^\ell\, e^{\beta_{j_s}}.
\]

\begin{lemma}\label{Ul1}
  Define $\bar\beta=\frac 1n\sum_{j=1}^n\beta_j$ and
  suppose that $\beta_j-\bar\beta=\OO(n^{-1/2})$
  uniformly for $j=1,\ldots,n$.
  Then, for sufficiently small $\eps>0$, we have
 \[
    U_\ell(\betavec) = \binom{n}{\ell}
           \exp\biggl( \ell\bar\beta
             + \frac{\ell(n-\ell)}{2n^2}
             	\sum_{j=1}^n(\beta_j-\bar\beta)^2
            + \OO(n^{-1/2}) \biggr),
 \]
  uniformly for $\ell=0,\ldots,n$.
\end{lemma}

\begin{proof}
The factor $e^{\ell\bar\beta}$
can be removed by replacing each $\beta_j$
by $\beta_j-\bar\beta$, so it suffices to prove the lemma
for $\sum_{j=1}^n\beta_j=\bar\beta=0$.

We divide the proof into three parts, depending on~$\ell$.
Let $B=\max_j\abs{\beta_j}$.  Choose a constant $c\ge 0$ such that
$Bn^{1/2-c\eps}=o(1)$.

First assume that $n^{1/2-c\eps}\le \ell\le n- n^{1/2-c\eps}$.
Since $U_\ell(\betavec)$ is the coefficient of $y^\ell$ in
$\prod_{j=1}^n (1 + e^{\beta_j} y)$, we can estimate it using the
saddle point method.  We choose the contour to be a circle of
radius $r$ centered at the origin, where
\[   r = \frac{\ell}{n-\ell}.  \]
For $j=1,\ldots, n$ let
\[ \psi_j = \frac{e^{\beta_j}r}{1+e^{\beta_j}r}. 
\]
Changing variable according to $y=re^{i\theta}$ and applying Cauchy's
theorem, we obtain 
\[ U_\ell(\betavec) = P(\betavec) \int_{-\pi}^{\pi} F(\theta)\,d\theta, \]
where
\[
   P(\betavec) = \frac{\prod_{j=1}^n (1+re^{\beta_j})}{2\pi r^\ell},
   \qquad
   F(\theta) = \frac{\prod_{j=1}^n\(1+\psi_j(e^{i\theta}-1)\)}{e^{i\ell\theta}}.
\]
The coefficient $\psi_j$ satisfies
\begin{align}
   \psi_j &= \frac{\ell}{n} + \frac{\ell(n-\ell)}{n^2} \, \beta_j
         + \frac{\ell(n-\ell)(n-2\ell)}{2n^3}\, \beta_j^2
               + \OO(\beta_j^3)\label{muj-equation}\\
	&= \frac{\ell}{n}\(1+\OO(n^{-1/2})\). \notag
\end{align}
We now divide the domain of integration into the two subdomains
$\abs\theta\le\theta_0$ and $\abs\theta>\theta_0$, where
\[
    \theta_0 = \sqrt{\frac{n}{\ell(n-\ell)}}\,\log n.
\]
Expanding $F(\theta)$ for $\abs\theta\le\theta_0$, we find 
using (\ref{muj-equation}) that
\begin{align*}
  F(\theta) &= \exp\biggl( -i\ell\theta + i\sum_{j=1}^n\psi_j\theta
   - \dfrac12\sum_{j=1}^n\psi_j(1-\psi_j)\,\theta^2 \\
   &{\kern20mm}- \dfrac16 i \sum_{j=1}^n\psi_j(1-\psi_j)(1-2\psi_j)\,\theta^3
   + O\(\ell(n-\ell)n^{-1}\theta_0^4\) \biggr) \\
  &=
  \exp\biggl( - \frac{\ell(n-\ell)}{2n} \theta^ 2
      + O(1)\,\frac{i \ell(n-\ell)}{n}\theta^3 + \OO(n^{-1/2}) \biggr),
\end{align*}
where the $O(1)$ term is independent of $\theta$.
Since the interval $\abs\theta\le\theta_0$ is symmetric about~0,
we can instead integrate
\[
  \dfrac12\(F(-\theta)+F(\theta)\)
   = \exp\biggl( - \frac{\ell(n-\ell)}{2n}\, \theta^ 2 + \OO(n^{-1/2}) \biggr).
\]
Furthermore,
\[ \int_{\theta_0}^\infty 
   \exp\biggl( - \frac{\ell(n-\ell)}{2n}\, \theta^ 2 + \OO(n^{-1/2}) \biggr) \, 
              d\theta
	      = n^{-\Omega(\log n)}\]
(and similarly for the lower tail) and hence
\[
   \int_{-\theta_0}^{\theta_0} F(\theta)\,d\theta
     = \sqrt{\frac{2\pi n}{\ell(n-\ell)}}\,\exp\(\OO(n^{-1/2})\).
\]
For the complementary subdomain $\abs\theta>\theta_0$, note that
\[ \abs{1+\psi_j(e^{i\theta}-1)} = \sqrt{1-2\psi_j(1-\psi_j)(1-\cos\theta)}\,, \]
which is a decreasing function for $\theta\in(\theta_0,\pi)$.
Therefore, $\abs{F(\theta)}\le\abs{F(\theta_0)}$ for $\abs\theta>\theta_0$.
Since $1-\cos y \geq 2y^2/\pi^2$ when $-\pi\leq y\leq \pi$,
we have
\begin{align*}
\abs{F(\theta_0)} &= \prod_{j=1}^n \sqrt{1-2\psi_j(1-\psi_j)(1-\cos\theta_0)}\\
  &\leq \exp\biggl( - \frac{2\log^2 n}{\pi^2 }  + \OO(n^{-1/2}) \biggr)\\
   &= n^{-\Omega(\log n)}.
\end{align*}
Hence 
\[
   \int_{-\pi}^{\pi} F(\theta)\,d\theta
  = 
    n^{-\Omega(\log n)}
   + \int_{-\theta_0}^{\theta_0} F(\theta)\,d\theta
     = \sqrt{\frac{2\pi n}{\ell(n-\ell)}}\,\exp\(\OO(n^{-1/2})\).
\]

Finally, we calculate that
\begin{align*}
   P(\betavec) &= \frac{n^n}{2\pi\ell^\ell(n-\ell)^{n-\ell}}
      \prod_{j=1}^n \frac{1+re^{\beta_j}}{1+r} \\
   &=
    \frac{n^n}{2\pi\ell^\ell(n-\ell)^{n-\ell}}
      \exp\biggl( \frac{\ell(n-\ell)}{2n^2}\sum_{j=1}^n\beta_j^2 
            + \OO(n^{-1/2}) \biggr).
\end{align*}
Therefore
\[
   U_\ell(\betavec)
    = \frac{n^{n+1/2}}{\sqrt{2\pi}\,\ell^{\ell+1/2}(n-\ell)^{n-\ell+1/2}}
       \exp\biggl( \frac{\ell(n-\ell)}{2n^2}\sum_{j=1}^n\beta_j^2
            + \OO(n^{-1/2}) \biggr),
\]
which equals the expression in the lemma, by Stirling's formula.

\medskip
We next consider the case that $0\le\ell<n^{1/2-c\eps}$.
Expand $U_\ell(\betavec)=\sum_{s\ge 0} T_s/s!$, where
\[
    T_s = \sum_{1\le j_1<\cdots<j_\ell\le n} 
               (\beta_{j_1}+\cdots+\beta_{j_\ell})^s.
\]
It follows from~\cite[Lemma 5]{CM} that
\begin{equation}\label{Tsmall}
T_0 = \binom{n}{\ell}, \quad T_1 = 0,\quad 
  T_2 = \binom{n}{\ell}\, O(\ell B^2) \quad \text{ and } \quad
  T_3 = \binom{n}{\ell}\, O(\ell B^3).
\end{equation}
We proceed to bound $T_s$ for $s\geq 4$.
Let $\sumppd_{j_1,\ldots, j_\ell}$ denote the sum over all sequences
$(j_1,\ldots,j_\ell)\in \{ 1,\ldots, n\}^\ell$
with $\ell$ distinct entries.
Applying the multinomial theorem, we have
\[
    T_s = \frac{1}{\ell!} \sum_{m_1+\cdots+m_\ell=s}
          \binom{s}{m_1,\ldots,m_\ell} B(m_1,\ldots,m_\ell),
\]
where
\[
   B(m_1,\ldots,m_\ell) = 
    \sumppd_{j_1,\ldots,j_\ell}
       \beta_{j_1}^{m_1}\cdots\beta_{j_\ell}^{m_\ell}.
\]

Let $\M_1$ be the set of all compositions
$\mvec=(m_1,\ldots,m_\ell)$  of $s$ such that $m_i=1$ for some~$i$,
and let $\M_2$ be the set of all other compositions of $s$.
For all $\mvec$ we have
\[ \abs{B(\mvec)} \le [n]_\ell B^s, \]
using the falling factorial.
For $\mvec\in\M_1$, suppose as a representative case that $m_\ell=1$.
Then 
\begin{align*}
  B(\mvec) &= \sumppd_{j_1,\ldots,j_\ell}
      \beta_{j_1}^{m_1}\cdots\beta_{j_\ell}^{m_\ell} \\
  &= \sumppd_{j_1,\ldots,j_{\ell-1}}
	\beta_{j_1}^{m_1}\cdots\beta_{j_{\ell-1}}^{m_{\ell-1}}
	  \sum_{j_\ell\notin\{j_1,\ldots,j_{\ell-1}\}} \beta_{j_\ell} \\
  &= -\sumppd_{j_1,\ldots,j_{\ell-1}}
        \beta_{j_1}^{m_1}\cdots\beta_{j_{\ell-1}}^{m_{\ell-1}}
	  \sum_{j_\ell\in\{j_1,\ldots,j_{\ell-1}\}} \beta_{j_\ell},
\end{align*}
where the last step uses the assumption $\sum_{j=1}^n\beta_j=0$.
This shows that for $\mvec\in\M_1$ we have
\[ \abs{B(\mvec)} \le \ell [n]_{\ell-1} B^s = O(\ell/n) [n]_\ell B^s. \]
Consequently
\[
  \abs{T_s} \le \binom{n}{\ell} \,B^s\, \biggl(
   O(\ell/n) \sum_{\mvec\in\M_1} \binom{s}{m_1,\ldots,m_\ell}
          + \sum_{\mvec\in\M_2} \binom{s}{m_1,\ldots,m_\ell}
   \biggr).
\]
Furthermore
\[
   \sum_{\mvec\in\M_1} \binom{s}{m_1,\ldots,m_\ell}\le \ell^s.\]
Next, notice that for any fixed integer $s\geq 4$,
\[
    C_s = \sum_{\mvec\in\M_2} \binom{s}{m_1,\ldots,m_\ell}
\]
is the coefficient of $x^s$ in the Maclaurin expansion of $s!\,(e^x-x)^\ell$.
Since that expansion has nonnegative coefficients,
$C_s\le s!\,\eta^{-s}(e^\eta-\eta)^\ell$ for any $\eta>0$.
Substituting $\eta=\sqrt{s/\ell}$ and using the fact that 
$(e^{\sqrt{x}} - \sqrt{x})^{1/x} < 2$ for $x>0$ gives
\[
   C_s \le s!\, (s/\ell)^{-s/2} 
            \( e^{\sqrt{s/\ell}}-\sqrt{s/\ell}\,\)^\ell
    \le s!\,\(2\sqrt{\ell/s}\,\)^s.
    \]
Hence we have, for any fixed integer $s\geq 4$,
\begin{equation}\label{S2}
\abs{T_s} \leq \binom{n}{\ell}\Bigl( O(\ell/n)\, \ell^s 
                + s!\,\(2\sqrt{\ell/s}\,\)^s\Bigr).
\end{equation}
Using \eqref{Tsmall} for $s\le 3$ and~\eqref{S2} for $s\ge 4$,
gives
\[  U_\ell(\betavec) = \binom{n}{\ell}
    \Bigl(1 + \OO(n^{-1/2})
       + O(\ell/n)\sum_{s\ge 4} \frac{1}{s!}B^s\ell^s
       + O(1)\sum_{s\ge 4} B^s\(2\sqrt{\ell/s}\,\)^s \Bigr). 
\]
Since $B\ell=o(1)$ and $B\sqrt{\ell} = \OO(n^{-1/4})$, 
the first sum in the above expression is $O(\ell/n) = \OO(n^{-1/2})$,
while the second sum is at most
\[ \sum_{s\geq 4} B^s\, \ell^{s/2} = O(B^4 \ell^2) = \OO(n^{-1}).\]
Hence
\[
 U_\ell(\betavec) = \binom{n}{\ell} \(1 + \OO(n^{-1/2})\),
\]
which matches the lemma for this range of~$\ell$ values.

For the remaining range $n-n^{1/2-c\eps}<\ell\le n$, we can apply the
identity $U_\ell(\betavec) = U_{n-\ell}(-\betavec)$, which is a
consequence of $\sum_j\beta_j=0$.
The lemma is thus proved.
\end{proof}

\nicebreak
\subsection{Proof of the dense theorem (Theorem~\ref{densetheorem})}
\label{ss:dense-proof}

Suppose that $a,b>0$ are constants such that $a+b<\frac12$, and
$\dvec$ is such that \eqref{dcondition} holds and $d_j-d$ is uniformly 
$O(n^{1/2+\eps})$ for $j=1,\ldots,n$ and some $\eps>0$.
In the following, we will assume that $\eps$ is sufficiently small.
Later in the proof we will infer that we can take $\eps=\eps(a,b)$ for some
function $\eps(a,b)>0$, as required by Theorem~\ref{densetheorem}.

Every vector $\zvec = (z_1,\ldots, z_n)\in \{ 0,1\}^n$ is a potential
diagonal of one of our matrices.   Define $\abs\zvec=\sum_{j=1}^n z_j$
and for $\ell=0,\ldots,n$ let 
\begin{equation} \Lambdait_\ell = \{ \zvec\in \Lambdait\,\, : 
               \,\, \abs\zvec  = \ell\}.
\label{lambda-ell-def}
\end{equation}
If $D\ell$ and $S$ have the same parity then
\begin{equation}
\label{dense-strategy}
   G_D(\dvec,\ell) = \sum_{\zvec\in\Lambdait_\ell} G(\dvec-D\zvec).
   \end{equation}
We proceed by applying Theorem~\ref{densenum}
to estimate $G(\dvec-D\zvec)$ and then summing the result over
all $\zvec\in\Lambdait_\ell$.   
Note that the average entry of $\dvec-D\zvec$ is $d - D\ell/n$. 

Let $\widehat{a}$ be
any constant such that $a < \widehat{a} < \nfrac{1}{2} -b$
and let $\varepsilon_0 = \varepsilon_0(\widehat{a},b)$ be the positive constant
guaranteed by Theorem~\ref{densenum}.
Then for $\ell=0,\ldots, n$ we have
\[ \min\,\biggl\{d-\frac{D\ell}{n},\, n - d + \frac{D\ell}{n} \biggr\} \geq 
              \frac{n}{3\widehat{a}\,\log n}\]
for sufficiently large $n$.
Provided $\eps\le\eps_0$, we have that 
$(d_j - Dz_j) - (d - D\ell/n)$ is uniformly $O(n^{1/2+\varepsilon_0})$ 
for $j=1,\ldots, n$.
So Theorem~\ref{densenum} with the constants $(\widehat{a},b)$
applies to every vector $\dvec - D\zvec$, using the value 
$\varepsilon_0 =\varepsilon_0(\widehat{a},b)$ guaranteed by that theorem.

Next we will compare factors from the expression for $G(\dvec - D\zvec)$
given by (\ref{Gdense}) with corresponding factors from
the formula for $G_D(\dvec,\ell)$ given in Theorem~\ref{densetheorem}.
Let $\lambda_\ell$ denote the density of $\dvec-D\zvec$ for any
$\zvec\in\Lambdait_\ell$.  That is,
\[   \lambda_\ell = \frac{d}{n-1} - \frac{D\ell}{n(n-1)}
           = \mu_D - \frac{D(\ell-\mu_D n)}{n(n-1)}. \]
Also let $\delta_j = d_j - d$ for $j = 1,\ldots, n$,
which allows us to write $R = \sum_{j=1}^n \delta_j^2$.
Then
\[
 \frac
  {\(\lambda_\ell^{\lambda_\ell}
    (1-\lambda_\ell)^{1-\lambda_\ell}\)^{\binom n2}}
  {\(\mu_D^{\mu_D} (1-\mu_D)^{1-\mu_D}\)^{\binom n2+Dn/2}}
  =
  \mu_D^{-D\ell /2}(1-\mu_D)^{-D(n-\ell)/2}
  \exp\biggl( \frac{D^2(\ell-\mu_D n)^2}{4\mu_D(1-\mu_D)n^2} + \OO(n^{-1}) \biggr).
\]
Using the expansion $[m]_k = m^k \exp\left(-\frac{k(k-1)}{2m} + O(k^3/m^2)\right)$,
valid when $m\to\infty$ such that $k = o(m^{2/3})$, we find that
\begin{align*}
\frac{\displaystyle\binom{n{-}1}{d_j{-}Dz_j}}
       {\displaystyle\binom{n{+}D{-}1}{d_j}}
         = \mu_D^{Dz_j}(1-\mu_D)^{D(1-z_j)}
  \exp\biggl(&
    {-\frac{D(D-1)(\mu_D-z_j)^2}{2\mu_D(1-\mu_D)n}}\\[-2.5ex]
     &{\kern-3mm}- \frac{D(\mu_D-z_j)\delta_j}{\mu_D(1-\mu_D)n} 
   -\frac{D(\mu_D-z_j)^2\delta_j^2}{2\mu_D^2(1-\mu_D)^2n^2}
      + \OO(n^{-3/2})
   \biggr).
\end{align*}
Since $z_j^2=z_j$ and $\sum_{j=1}^n z_j = \ell$, we obtain
\begin{align*}
 \prod_{j=1}^n
  \frac{\displaystyle\binom{n{-}1}{d_j{-}Dz_j}}
       {\displaystyle\binom{n{+}D{-}1}{d_j}}
 &= \mu_D^{D\ell}(1-\mu_D)^{D(n-\ell)} \notag\\[-2.5ex]
 &{\kern5mm}\times \exp\biggl( -\frac{D(D-1)}{2}
     - \frac{D(D-1)(1-2\mu_D)(\ell-\mu_D n)}{2\mu_D(1-\mu_D)n}
     - \frac{DR}{2(1-\mu_D)^2n^2}\notag\\
     &{\kern21mm}+ 
        \sum_{j=1}^n\biggl(
       \frac{D\delta_j}{\mu_D(1-\mu_D)n} 
      - \frac{D(1-2\mu_D)\delta_j^2}{2\mu_D^2(1-\mu_D)^2n^2} \biggr)z_j
     + \OO(n^{-1/2})
     \biggr).
\end{align*}
Finally, apart from the $O(n^{-b})$ error term,
the expression inside the exponential in~\eqref{Gdense} for $\dvec-D\zvec$  is
\[
   \frac14-\frac{\(\,\sum_{j=1}^n (d_j-Dz_j-\lambda_\ell(n-1))^2\)^2}
     {4\lambda_\ell^2(1-\lambda_\ell)^2n^4}
   = \frac14-\frac{R^2}{4\mu_D^2(1-\mu_D)^2n^4} + \OO(n^{-1/2}).
   \]
Combining these expressions gives
\begin{equation}
\label{dense-answer}
 G(\dvec-D\zvec) = A\,\,  V(\ell)
                    \, \exp\biggl( \sum_{j=1}^n \beta_j z_j\biggr)
\end{equation}
where
\begin{align}
  A &= \sqrt{2}\, \(\mu_D^{\mu_D}\, (1-\mu_D)^{1-\mu_D}\)^{\binom{n}{2} + Dn/2}
            \,\prod_{j=1}^n \binom{n{+}D{-}1}{d_j}
   \exp\( O(n^{-b}) + \OO(n^{-1/2}) \), \label{indep}\\
  V(\ell) &=  \mu_D^{\ell D/2}(1-\mu_D)^{(n-\ell)D/2}
      \exp\biggl(\frac{1}{4} - \frac{D(D-1)}{2} - 
  \frac{D(D-1)(1-2\mu_D)(\ell-\mu_D n)}{2\mu_D(1-\mu_D)n} \notag \\
  & \hspace*{47mm} {}  +
   \frac{D^2(\ell-\mu_D n)^2}{4\mu_D(1-\mu_D)n^2} - \frac{DR}{2(1-\mu_D)^2 n^2}
         - \frac{R^2}{4\mu_D^2(1-\mu_D)^2 n^4}
            \biggr), \label{Vdef} \notag\\
    \beta_j &= \frac{D\delta_j}{\mu_D(1-\mu_D)n}
              - \frac{D(1-2\mu_D)\delta_j^2}{2\mu_D^2(1-\mu_D)^2n^2}
       \qquad  \mbox{ for $j=1,\ldots, n$}. \notag
\end{align}

Next we must sum over all $\zvec\in\Lambdait_\ell$. 
Note that $\beta_j = \OO(n^{-1/2})$ for $j=1,\ldots, n$,
and the average of $\beta_1,\ldots, \beta_n$ is
\[ \bar\beta = -\frac{D(1-2\mu_D)R}{2\mu_D^2(1-\mu_D)^2 n^3} = \OO(n^{-1}).\]
Hence Lemma~\ref{Ul1} applies and shows that
\begin{align}
 \sum_{\zvec\in\Lambdait_\ell}  \exp\biggl(\sum_{j=1}^n\beta_jz_j\biggr)
  &= \binom{n}{\ell}\,\exp\biggl(
 \ell\bar\beta+\frac{\ell(n-\ell)}{2n^2}
   \sum_{j=1}^n (\beta_j-\bar\beta)^2 + \OO(n^{-1/2})\biggr) \notag\\
    &= \binom{n}{\ell}\, \exp\biggl(\frac{D^2\ell(n-\ell)R}{2\mu_D^2(1-\mu_D)^2n^4}
       - \frac{D(1-2\mu_D)\ell R}{2\mu_D^2(1-\mu_D)^2n^3}
       + \OO(n^{-1/2})\biggr).
\label{betabit}
\end{align}
Combining (\ref{dense-strategy}) and (\ref{dense-answer})--(\ref{betabit}) gives
\begin{equation}
   G_D(\dvec,\ell) = A\, \binom{n}{\ell} \mu_D^{\ell D/2}
              (1-\mu_D)^{(n-\ell)D/2}\, \exp\( Q_D(\dvec,\ell) + \OO(n^{-1/2})
  \)
\label{ell-equation}
\end{equation}
where $Q_D(\dvec,\ell)$ is defined in the statement of Theorem~\ref{densetheorem}
for $D\in\{1,2\}$.

Next we will estimate $G_D(\dvec)$ by
summing (\ref{ell-equation}) over allowable values of $\ell$.
Recall the definition of $\bar\ell_D$ given in the theorem statement.
Ignoring the factor $A$
which is independent of $\ell$, we calculate
\begin{align}
 &  \sum_{\ell = 0}^n \binom{n}{\ell}\mu_D^{D\ell/2}(1-\mu_D)^{D(n-\ell)/2}
            \exp\(Q_D(\dvec,\ell) + \OO(n^{-1/2})\) \notag \\
 &=  (1-\mu_D)^{Dn/2}\,
    \sum_{\ell = 0}^n \binom{n}{\ell}\biggl(\frac{\mu_D}{1-\mu_D}\biggr)^{\!D\ell/2}
            \exp\(Q_D(\dvec,\bar\ell_D) +
              \OO(n^{-1}(\ell-\bar\ell_D) + n^{-1/2})\) \notag \\
 &=  (1-\mu_D)^{Dn/2}\, \exp(Q_D(\dvec,\bar\ell_D))\,
    \sum_{\ell = 0}^n \binom{n}{\ell}\biggl(\frac{\mu_D}{1-\mu_D}\biggr)^{\!D\ell/2}
            \exp\(
              \OO(n^{-1}(\ell-\bar\ell_D) + n^{-1/2})\). \label{binexp}
\end{align}
If $\abs{\ell - \bar\ell_D}\leq n^{1/2+\eta}$ for some constant
$\eta>0$ then the error term in the corresponding summand is $\OO(n^{-1/2})$,
so these summands are essentially terms from a binomial expansion.
If $\abs{\ell - \bar\ell_D} > n^{1/2 + \eta}$ then
\begin{equation}\label{bintail}
  \binom{n}{\ell} \biggl(\frac{\mu_D}{1-\mu_D}\biggr)^{D\ell/2}
         \leq \exp(-\Omega(n^{2\eta})),\qquad
   \exp\(\OO(n^{-1}(\ell - \bar\ell_D))\) = \exp(\OO(1)),
\end{equation}
so the contribution from the tails of the sum is negligible.
Therefore
\begin{align} 
 \sum_{\ell = 0}^n \binom{n}{\ell} &
  \biggl(\frac{\mu_D}{1-\mu_D}\biggr)^{\!D\ell/2}
   \, \exp\biggl(\OO\biggl(\frac{\ell-\bar \ell_D}{n} + n^{-1/2}\biggr)\biggr)
            \notag\\
   &= 
   \exp(\OO(n^{-1/2}))\,
   \sum_{\ell = 0}^n \binom{n}{\ell} 
  \biggl(\frac{\mu_D}{1-\mu_D}\biggr)^{\!D\ell/2}
     \label{uniform}\\
  &= \biggl(1 + \biggl(\frac{\mu_D}{1-\mu_D}\biggr)^{\!D/2\,}\biggr)^{\!n} 
                          \exp\(\OO(n^{-1/2})\).  \label{uniform2}
\end{align}

The preceding calculations
hold for any sufficiently small $\eps>0$, so in particular
they hold for some $\eps=\eps(a,b)$ such that $\eps\le\eps_0$
and the $\OO(n^{-1/2})$ error terms in
(\ref{indep}), (\ref{ell-equation}) and (\ref{uniform2}) are all~$O(n^{-b})$.
Then the claimed formulae for 
$G_D(\dvec,\ell)$
follow immediately from (\ref{indep}) and~(\ref{ell-equation}).

Furthermore, multiplying (\ref{uniform2}) by
$A\, (1-\mu_D)^{Dn/2}\, \exp(Q_D(\dvec,\bar\ell_D))$ using (\ref{indep})
and substituting $D=2$
gives the desired formula for $G_2(\dvec)$. 

For $D=1$, we must sum over only those values of $\ell$ with
the same parity as $S$.  That is, we must replace (\ref{uniform})
with a sum over just the even (or just the odd) values of $\ell$.
By standard properties of the binomial distribution, the parity-restricted
sum is half the full sum, within additive error
$O(n^{-b})$, say.  (This also follows from Lemma~\ref{parity} when
$\mu_1\neq \nfrac{1}{2}$, and hence when $\mu_1=1/2$ by analytic continuation.)
The additive error can be absorbed into the relative error in (\ref{uniform2}),
since the main factor there is $\Omega(1)$.
This gives the desired formula for $G_1(\dvec)$, 
completing the proof.
\qed

\section{The sparse case}\label{s:sparse}

In this section we prove Theorem~\ref{sparsetheorem}.

\subsection{Some useful results}\label{ss:sparse-technical}

First, we present two lemmas involving the Poisson binomial
distribution, which we introduced in Section~\ref{ss:distribution}.
Let $\pvec=(p_1,\ldots,p_n)$ satisfy $0\leq p_1,\ldots, p_n\leq 1$
and let $X$ be a random variable with distribution~$\PB(\pvec)$.
The mean of $X$ is $\bar{X} = \E(X) = \sum_{j=1}^n p_j$.
The following tail bounds are standard.

\begin{lemma}
\label{Chernoff}
If $X$ is a Poisson binomial random variable then,
for any $s\ge 0$, we have
\begin{align*}
 \Pr(X - \bar{X} \leq -s) &\leq 
   \exp\biggl(- \frac{s^2}{2\bar{X}}\biggr),\\
\Pr\(X- \bar{X} \geq s\) &\leq 
       \exp\biggl( -\bar{X} \varphi\biggl(\frac{s}{\bar{X}}\biggr)\biggr),
\end{align*}
where $\varphi(x)=(1+x)\log(1+x)-x$.
\end{lemma}

\begin{proof}
These bounds are attributed to Chernoff, see
\cite[Theorems 2.1 and 2.8]{JLR}
and~\cite[Theorem 3.2]{CL}.
\end{proof}

\begin{lemma}
\label{technical}
Let $X$ be a random variable with Poisson binomial
distribution~$\PB(\pvec)$ and mean $\bar{X} \le n/(\log n)^2$.
For a fixed constant $C>0$, let
$f:\Reals\rightarrow\Reals$ be a function such that
\[ \abs{f(x)} \leq C\biggl(\frac{x^2}{n}  +
         \frac{\abs x}{n^{1/2}}\biggr)
          \quad \text{ for } \quad \abs x\leq n.\]
Then
\begin{align}
 \E\(\exp(f(X-\bar{X}))\) &= 1 + \E\(f(X-\bar{X})\) + 
     O\(\E(f(X-\bar{X})^2\,)\) + n^{-\Omega(\log n)} \notag \\
   &= \exp\Bigl(\E\(f(X-\bar{X})\) + O\(\E(f(X-\bar{X})^2\,)\)
      + n^{-\Omega(\log n)} \Bigr).
\label{first}
\end{align}
In particular, the contribution to this expectation from values of
$X$ with $\abs{X-\bar{X}} > n^{1/2}$ is
\begin{equation}
\sum_{\ell,\, \abs{\ell - \bar{X}} > n^{1/2}}  \Pr(X=\ell) \exp\(f(\ell - \bar{X})\)
  = n^{-\Omega(\log n)}.
\label{contribution}
\end{equation}
\end{lemma}

\begin{proof}
Define $g(x) = e^{f(x)} - 1 - f(x)$.  
Note that
 $\abs{g(x)} \leq e^{\abs{f(x)}} $
for all $x$, which implies that
\begin{equation}
\abs{g(x)} \leq \exp\biggl(\frac{C\, x^2}{n} + \frac{C\, \abs x}{n^{1/2}}\biggr)
          \quad \text{ for } \quad \abs{x}\leq n.
\label{Deltabound}
\end{equation}
We write
\[ \E\(g(X-\bar{X})\) = \Sigma_1 + \Sigma_2 + \Sigma_3\]
where
\begin{align*}
\Sigma_1 &= 
           \sum_{\ell,\, \abs{\ell - \bar{X}} \leq n^{1/2}} \Pr(X = \ell)\, g(\ell - \bar{X}),\\
\Sigma_2   &= \sum_{\ell,\, \ell - \bar{X} > n^{1/2}} \Pr(X = \ell)\, g(\ell - \bar{X}),\\
\Sigma_3 &= \sum_{\ell,\, \ell - \bar{X} < -n^{1/2}} \Pr(X = \ell)\, g(\ell - \bar{X}).
\end{align*}
In each of these sums, $\ell$ is a nonnegative integer in $\{ 0,\ldots, n\}$ which
satisfies the additional constraint given.  We now bound these three sums in turn.

For $\Sigma_1$, note that when
$\abs{\ell - \bar{X}}\leq n^{1/2}$ we have  
$f(\ell - \bar{X})\leq 2C = O(1)$.   Hence 
\[ g(\ell - \bar{X}) = O\( f(\ell - \bar{X})^2 \)\]
uniformly for all $\ell$ in this range.
It follows that
$\sum_{\ell,\, \abs{\ell - \bar{X}} \leq n^{1/2}} \Pr(X = \ell)\, f(\ell - \bar{X})^2
 \leq \E(f(X-\bar{X})^2)$, and hence
\begin{equation}
\label{firstcase}
 \Sigma_1 =  O\(\E(f(X - \bar{X})^2)\).
\end{equation}

Now we consider $\Sigma_2$.  
Since $\bar X\varphi(s/\bar X)$ is a decreasing function of $\bar X$,
and $\bar{X}\leq n/(\log n)^2$ by assumption, Lemma~\ref{Chernoff}
shows that 
\[ \Pr\(X- \bar{X} \geq s\) \leq 
       \exp\biggl( -\frac{n}{\log^2 n} 
              \varphi\biggl(\frac{s\log^2 n}{n}\biggr)\biggr).\]
Applying (\ref{Deltabound}) shows that  $\Sigma_2$ is bounded above by
\[
   n \, \max_{n^{1/2} < s \leq n} \exp\( L_1(s) \),
\]
where
\[ L_1(s) = \frac{Cs^2}{n} + \frac{Cs}{n^{1/2}}
             -
       \frac{n}{\log^2 n} \,
       \varphi \biggl(\frac{s\log^2 n}{n}\biggr).
\]
Now 
\[ L_1'(s) = \frac{2Cs}{n} + \frac{C}{n^{1/2}} - 
        \log\biggl(1 + \frac{s\log^2 n}{n}\biggr),\]
which is negative for sufficiently large $n$ for
$s\in\{n^{1/2},n\}$.
Also $L_1'''(s) >0$ for all $s\ge 0$, so it must
be that $L_1'(s)<0$ for $n^{1/2}\le s\le n$ when $n$ is 
sufficiently large.  
It follows that the maximum of $L_1$ on the interval
$[n^{1/2},n]$ occurs at $s=n^{1/2}$.  Since
$L_1(n^{1/2}) = -\nfrac{1}{2}\log^2 n + O(1)$,
we deduce that 
\begin{equation}
\label{secondcase}
\Sigma_2 = n\, \exp\( -\Omega(\log^2 n)\) = n^{-\Omega(\log n)}.
\end{equation}

A bound on $\Sigma_3$ can be obtained similarly.
Using the first bound in Lemma~\ref{Chernoff}, we find
\[ \Sigma_3 \le n \max_{n^{1/2}< s \leq n} \exp\(L_2(s)\) ,\]
where
\[ L_2(s) = \frac{C s^2}{n} + \frac{C s}{n^{1/2}} -\frac{ s^2\log^2n}{2n}.\]
By the same argument as before, the maximum of $L_2(s)$ occurs
at $s=n^{1/2}$ for sufficiently large~$n$, and we
conclude that $\Sigma_3 = n^{-\Omega(\log n)}$, 
which together with (\ref{secondcase}) implies (\ref{contribution}). 
Combining (\ref{contribution}) and (\ref{firstcase})
establishes (\ref{first}), completing the proof.
\end{proof}

For a given function $f:\{0,1,\ldots, n\}\to\Reals$, define the
polynomial $\hat{f}:\Reals^n\rightarrow\Reals$ by
\[
\hat{f}(y_1,\ldots, y_n) = \sum_{(x_1,\ldots, x_n)\in \{ 0,1\}^n}
   f(x_1 + \cdots + x_n)\, \prod_{j=1}^n y_j^{x_j}\, (1-y_j)^{1-x_j}.
\]
(Note that this is indeed a polynomial in $y_1,\ldots, y_n$,
since $1-x_j\in\{ 0,1\}$.)
In the case that $0\le y_1,\ldots,y_n\le 1$, we have
\begin{equation}
\label{fhatEY}
 \hat{f}(y_1,\ldots, y_n) = \E\(f(Y)\),
\end{equation}
where $Y$ is a random variable with distribution $\PB\((y_1,\ldots,y_n)\)$.

The following lemma will be used when $D=1$ to handle the
parity restriction on the number of loops.

\begin{lemma}
\label{parity}
Fix $(p_1,\ldots, p_n)\in [0,1]^n$ such that $p_j\neq \nfrac{1}{2}$
for $j=1,\ldots, n$.  Define
\[ r_j = -\frac{p_j}{1-2p_j}\]
for $j=1,\ldots, n$, and let
\[ Z = \prod_{j=1}^n \,(1 - 2p_j).\]
Then for $\rho = 0,\, 1$,
\begin{align*}
\sum_{\stackrel{(x_1,\ldots, x_n)\in \{0,1\}^n}
                   {x_1 + \cdots + x_n \equiv \rho\!\!\pmod{2}}}
   f(x_1 + \cdots & + x_n)\,  \prod_{j=1}^n \,p_j^{x_j}\, (1-p_j)^{1-x_j}\\
   &	= \nfrac{1}{2} \, \hat{f}(p_1,\ldots, p_n) + 
		       (-1)^\rho\,\nfrac{1}{2} \, Z\, \,
              \hat{f}(r_1,\ldots, r_n).
\end{align*}
\end{lemma}

\begin{proof}
Let $X$ be a random variable with Poisson binomial
distribution $\PB\((p_1,\ldots,p_n)\)$.
The probability generating function for $X$ is
\[ P(w) = \sum_{t=0}^n w^t \Pr(X=t) = \prod_{j=1}^n \,(1 - p_j + p_j w).\]
Note that
\begin{equation}
\label{fextract}
 \hat{f}(p_1,\ldots, p_n) = \sum_{t=0}^n f(t)\, \Pr(X=t) = 
                                   \sum_{t=0}^n f(t)\, [w^t] P(w).
\end{equation}
Now
\[ P(-w) = \prod_{j=1}^n \,(1-p_j-p_jw) = Z\, \prod_{j=1}^n \,(1-r_j + r_j w).\]
This expression has the same algebraic form as $P(w)$, but with $r_j$
in place of $p_j$ for $j=1,\ldots, n$.  Therefore, by comparison with
(\ref{fextract}) we have
\[ \sum_{t=0}^n f(t)\, [w^t] P(-w) = Z\, \hat{f}(r_1,\ldots, r_n).\]
Hence we calculate that
\begin{align*}
 \sum_{\stackrel{(x_1,\ldots, x_n)\in\{ 0,1\}^n}
                 {x_1+\cdots + x_n \equiv\rho\!\! \pmod{2}}}
	f(x_1 + \cdots & + x_n)\,  \prod_{j=1}^n \, p_j^{x_j}(1-p_j)^{1-x_j}\\
	 &= \sum_{\stackrel{t=0,\ldots, n}{t\equiv \rho\!\! \pmod{2}}} 
	                  f(t)\Pr(X=t) \displaybreak[0]\\
	 &= \sum_{\stackrel{t=0,\ldots, n}{t\equiv \rho\!\! \pmod{2}}} f(t)\, 
	                         [w^t] \(\nfrac{1}{2} P(w) 
				          + (-1)^\rho \nfrac{1}{2} P(-w)\) \displaybreak[0]\\
	 &= \sum_{t=0}^n f(t)\, 
	                         [w^t] \(\nfrac{1}{2} P(w) 
				          + (-1)^\rho \nfrac{1}{2} P(-w)\)\\
         &= \nfrac{1}{2}\hat{f}(p_1,\ldots, p_n) + 
	         \nfrac{1}{2}(-1)^\rho\, Z\, \hat{f}(r_1,\ldots, r_n),
\end{align*}
as claimed.
\end{proof}

\subsection{Proof of the sparse theorem (Theorem~\ref{sparsetheorem})}
\label{ss:proofs-sparse}

We now prove Theorem~\ref{sparsetheorem}.
Assume throughout this section that $1\leq \dmax = o(S^{1/3})$
and that $S$ is even if $D=2$. Furthermore, note that
deleting vertices of degree zero does not affect either
the value of $G_D(\dvec)$ or the formulae for it given in
Theorem~\ref{sparsetheorem}. Hence we assume without loss of
generality that $d_j\geq 1$ for $j=1,\ldots, n$.

Let
\[
 H(\dvec) = 
 \sqrt{2}\, \biggl(\frac{S}{e}\biggr)^{\!S/2}
  \, \biggl(\prod_{j=1}^n d_j!\biggr)^{\!-1}
  \exp\biggl( -\frac{S_2}{2S} - \frac{S_2^2}{4S^2} - \frac{S_2^2 S_3}{2S^4}
   + \frac{S_2^4}{4S^5} + \frac{S_3^2}{6S^3}\biggr).
   \]
Using Stirling's approximation, Theorem~\ref{sparsenum} can be restated
as follows: when $S$ is even and $1\leq \dmax = o(S^{1/3})$, then
\begin{equation}
G(\dvec) = H(\dvec)\exp\( O(\dmax^3/S)\),
\label{MWexpression}
\end{equation}
uniformly as $S\to\infty$.  
We proceed to estimate $G_D(\dvec)/H(\dvec)$.

Define
\begin{align*}
\Lambdait^{(1)} &= \{ \zvec\in\{ 0,1\}^n :
                \text{ for $j=1,\ldots, n$, if $d_j < D$ then $z_j = 0$}\},\\ 
\Lambdait^{(2)} &= \{ \zvec\in \{0,1\}^n :
                     \card{\zvec} \equiv S \!\!\pmod{2}\}
\end{align*}
and let
\[ \Lambdait= \begin{cases} 
   \Lambdait^{(1)}\cap\Lambdait^{(2)} & \text{ if $D=1$, }\\[0.6ex]
\Lambdait^{(1)} & \text{ if $D=2$.}
\end{cases}
\]
(Recall that $\card{\zvec}$
denotes the number of entries of $\zvec$ equal to 1.)
Then
\begin{equation}
\label{strategy}
 G_D(\dvec) = H(\dvec) \,
          \sum_{\zvec\in\Lambdait} 
      \frac{G(\dvec - D\zvec)}{H(\dvec)}.
\end{equation}
Our strategy is to compare the ratio 
$G(\dvec - D\zvec)/H(\dvec)$ to
the ratio $H(\dvec - D\zvec)/H(\dvec)$,
which we now investigate.   

\begin{lemma}
\label{small-ell}
For $j=1,\ldots,n$, define
\[ a_j = \frac{[d_j]_D}{S^{D/2}}\, \exp( \Deltait +  \gamma_j ), \]
where
\[ \gamma_j = - \frac{D(D+1)}{2S} - \frac{D(D+2)S_2}{2S^2} 
    - \frac{DS_2^2}{2S^3} 
        + \biggl( \frac{D}{S} + \frac{D S_2}{S^2}\biggr) d_j
\] 
for $j=1,\ldots, n$,
and $\Deltait$ is defined by
\[   \Deltait = \begin{cases}
        \frac{1}{2S^{1/2}} - \frac{S_2}{2S^2} & \text{ if $D=1$,}\\
        0 & \text{ if $D=2$.}\end{cases}
\]
Define $K(\zvec)$ by
\begin{equation}
\label{answer}
 \frac{H(\dvec - D\zvec)}{H(\dvec)}
        = 
  \exp\( K(\zvec) \) 
            \prod_{j=1}^n a_j^{z_j}.
\end{equation}
Then there are functions $K',K'': \{0,1,\ldots,n\}\to\Reals$ 
which satisfy
\begin{equation}
\label{Kdef}
K'(\ell),K''(\ell) =  - \Deltait \ell +
\frac{D^2\ell^2}{4S} + \frac{D^3\ell^3}{12 S^2}
 + O\biggl(
          \frac{\dmax^2 \ell^2}{S^2} + \frac{\ell^4}{S^3}
         + \frac{\dmax^3}{S} \biggr)
\end{equation}
such that
\begin{equation}
\label{Kbounds}
  K'(\card{\zvec}) \le K(\zvec) \le K''(\card{\zvec}) 
\end{equation}
for all $\zvec\in\Lambdait$ with $\card\zvec\le S/3$.
\end{lemma}

\begin{proof}
Define the function
\begin{align*}
M(\dvec,\zvec) &=
           -\frac{S_2(\zvec)}{2S_1(\zvec)} + \frac{S_2}{2S}
       - \frac{S_2(\zvec)^2}{4S_1(\zvec)^2} +
       \frac{S_2^2}{4S^2}
       - \frac{S_2(\zvec)^2S_3(\zvec)}{2S_1(\zvec)^4}
       + \frac{S_2^2S_3}{2S^4}  \\
       & \hspace*{1cm} {}
       + \frac{S_2(\zvec)^4}{4S_1(\zvec)^5} -
       \frac{S_2^4}{4S^5}
       + \frac{S_3(\zvec)^2}{6S_1(\zvec)^3} -
       \frac{S_3^2}{6S^3},
\end{align*}
where $S_r(\zvec) = \sum_{j=1}^n [d_j-Dz_j]_r$ for $r= 1, 2, 3$.
Then 
\begin{align*}
\label{eq1}
\frac{H(\dvec - D\zvec)}{H(\dvec)}
  &= 
  \exp({M(\dvec,\zvec)})\,
\biggl(\frac{e}{S}\biggr)^{\!D\ell/2} \, 
  \biggl(1 - \frac{D\ell}{S}\biggr)^{\!(S-D\ell)/2}\, 
       \prod_{j=1}^n \,([d_j]_{D})^{z_j} \\
  &= 
  \exp({M(\dvec,\zvec)})\,
 \exp\biggl(\frac{D^2\ell^2}{4S} + \frac{D^3\ell^3}{12 S^2}
   + O\biggl(\frac{\ell^4}{S^3}\biggr)\biggr)\,
\prod_{j=1}^n \biggl(\frac{[d_j]_D}{S^{D/2}}\biggr)^{\!z_j}.
\end{align*}
Now
\begin{align*}
S_1(\zvec) &= S - D\ell,\\
S_2(\zvec)  &= S_2  - 2D W_1  + D(D+1)\ell,  \\
S_3(\zvec)  &= S_3 - 3D W_2 + 3D(D+1)W_1 - D(D+1)(D+2) \ell,
\end{align*}
where $W_r = \sum_{j=1}^n [d_j]_r \, z_j$ for $r=1,2$.
Making these substitutions gives
\[
M(\dvec,\zvec)
    = 
 O\biggl(\frac{\dmax^3}{S} + \frac{\dmax^2 \ell^2}{S^2}
                  \biggr) +  
 \sum_{j=1}^n \gamma_j z_j.  
\]
Since the terms involving $\Deltait$ cancel, this completes the proof.
The lemma is in fact true for any $\Deltait$, but the value we have chosen
will be useful in proving Lemma~\ref{Kexpectation}.
\end{proof}

We now calculate some important quantities which will be needed later.

\begin{lemma}
\label{useful}
When $D=1$,
\begin{align*}
\sum_{j=1}^n \frac{a_j}{1 + a_j} &=  \sqrt{S} - \frac{1}{2} +
            \frac{1}{8S^{1/2}} - \frac{S_2}{S} + \frac{2S_2}{S^{3/2}}
        + \frac{S_3}{S^{3/2}} + \frac{S_2^2}{2S^{5/2}}
     + O\biggl(\frac{\dmax^3}{S}\biggr), \\
\sum_{j=1}^n  \log(1 + a_j ) &=  
         \sqrt{S} - \frac{1}{24S^{1/2}} -\frac{S_2}{2S} 
   + \frac{S_2}{2S^{3/2}} + \frac{S_3}{3S^{3/2}} 
          + \frac{S_2^2}{2S^{5/2}}
            + O\biggl(\frac{\dmax^3}{S}\biggr).
\end{align*}
When $D= 2$,
\begin{align*} \sum_{j=1}^n \frac{a_j}{1+a_j} &= 
    \frac{S_2}{S} \exp\(O(\dmax^2/S)\),\\
  \sum_{j=1}^n \log(1+a_j) &= 
    \frac{S_2}{S}\exp\(O(\dmax^2/S)\).
\end{align*}
\end{lemma}

\begin{proof}
For $D=1$ we have
\[ \Deltait = O(S^{-1/2}),\quad 
 \gamma_j = O\biggl(\frac{\dmax^2}{S}\biggr)+O\biggl(\frac{\dmax}{S}\biggr) d_j,
\]
and find that
\begin{align}
\label{asums}
\begin{split}
  \sum_{j=1}^n a_j &= 
    \sqrt{S} + \frac{1}{2} + \frac{1}{8S^{1/2}} + \frac{S_2^2}{2S^{5/2}}
     + O\biggl(\frac{\dmax^3}{S}\biggr), \\
  \sum_{j=1}^n a_j^2 &=
    \frac{S_2}{S} + \frac{S_2}{S^{3/2}} + 1 + \frac{1}{S^{1/2}} 
             + O\biggl(\frac{\dmax^3}{S}\biggr), \\
  \sum_{j=1}^n a_j^3 &= 
    \frac{S_3}{S^{3/2}} + \frac{3S_2}{S^{3/2}} + \frac{1}{S^{1/2}} 
               + O\biggl(\frac{\dmax^3}{S}\biggr), \\
  \sum_{j=1}^n a_j^4 &= O\biggl(\frac{\dmax^3}{S}\biggr),
\end{split} 
\end{align}
from which the result follows. When $D=2$ we have
\[ 
\sum_{j=1}^n a_j =  \frac{S_2}{S}
                 + O\biggl(\frac{\dmax^2S_2}{S^2}\biggr),\qquad
 \sum_{j=1}^n a_j^2 =  O\biggl(\frac{\dmax^2S_2}{S^2}\biggr),
\]
which imply the result in this case.
\end{proof}

Next we calculate the sum of the right hand side of
(\ref{answer}) over all $\zvec\in\{ 0,1\}^n$
(subject to a parity constraint if $D=1$),
after dividing by the factor $\prod_{j=1}^n (1 + a_j)$.

\begin{lemma}
\label{Kexpectation}
Let $K^*$ be either of the functions $K',K''$ defined in Lemma~\ref{small-ell}.\\
If $D=1$ then for $\rho \in \{ 0,1\}$,
\[
  \sum_{\stackrel{\zvec\in\{ 0,1\}^n}{\abs{\zvec}\equiv\rho\text{~\rm (mod 2)}}}
		     \kern-0.5em\exp\(K^*(\card\zvec)\)\,
   \prod_{j=1}^n \frac{a_j^{z_j}}{1+a_j} 
          = \dfrac{1}{2}\, \exp\biggl(
     -\frac{1}{4} + \frac{1}{3S^{1/2}}  + \frac{S_2}{2S^{3/2}}
                            + O\biggl(\frac{\dmax^3}{S}\biggr)
           \biggr).
\]
If $D=2$ then
\[ \sum_{\zvec\in\{ 0,1\}^n} \exp\(K^*(\card\zvec)\) \,
   \prod_{j=1}^n \frac{a_j^{z_j}}{1+a_j} 
     = \exp\( O(\dmax^3/S)\).
\]
\end{lemma}

\begin{proof}
Define $\pvec=(p_1,\ldots,p_n)$ where $p_j = a_j/(1+a_j)$
for $j=1,\ldots, n$, and 
let $X$ be a random variable with Poisson binomial
distribution $\PB(\pvec)$.
Then
\[ \sum_{\zvec\in\{ 0,1\}^n} \exp(K^*(\card{\zvec}))\,
   \prod_{j=1}^n \frac{a_j^{z_j}}{1+a_j} 
   = \E\(\exp(K^*(X))\). \]
The expectation of $X$ is $\bar{X} = \sum_{j=1}^n p_j$,
which has been calculated for $D=1,2$ in Lemma~\ref{useful}.

First suppose that $D=1$.
Recall from Lemma~\ref{useful} that
\[ \sum_{j=1}^n p_j =\sqrt S + O(\dmax) = \sqrt S + o(S^{1/3}).\]
{}From (\ref{contribution}) we know that 
\[
   \sum_{\abs\zvec>3\sqrt S}
 \exp(K^*(\abs\zvec))\prod_{j=1}^n \frac{a_j^{z_j}}{1+a_j} = n^{-\Omega(\log n)}.
\]
Next we observe that by Lemma~\ref{Chernoff},
\[
   \sum_{\abs{\sqrt S-\abs{\zvec}}>S^{1/3}} \; 
       \prod_{j=1}^n \frac{a_j^{z_j}}{1+a_j} = O(e^{-S^{1/6}}).
\]
For $\ell-\sqrt S = O(S^{1/3})$ we have, using (\ref{Kdef}),
\begin{equation}
\label{Kexpand}
 \exp\(K^*(\ell)\) = 
    \exp\(-\dfrac{1}{4} + O(\dmax^3/S)\) f(\ell), 
\end{equation}
where
\begin{align*}
   f(\ell) &= \frac{\ell^4}{32S^2} - \frac{\ell^3}{8S^{3/2}} + 
         \biggl( \frac{7}{16S} + \frac{S_2}{8S^{5/2}} + 
                   \frac{13}{48S^{3/2}}\biggr)\ell^2 \\
         &{\quad}+\biggl( -\frac{5}{8S^{1/2}} + \frac{S_2}{4S^2} 
                        - \frac{7}{24S}\biggr)\ell
         + \frac{41}{32} +\frac{S_2}{8S^{3/2}} + \frac{5}{48S^{1/2}}\;.
\end{align*}
Since $K^*(\ell)=O(1)$ for $\ell=O(\sqrt{S}\,)$, it follows that
\begin{align}
    \sum_{\abs\zvec\equiv \rho\!\!\pmod{2}}&
   \kern-0.5em\exp(K^*(\abs\zvec))\,\prod_{j=1}^n \frac{a_j^{z_j}}{1+a_j} \notag\\
   &= 
		      n^{-\Omega(\log n)} + 
    \exp\(-\dfrac{1}{4} + O(\dmax^3/S)\)
      \sum_{\abs\zvec\equiv \rho\text{~\rm (mod 2)}}
      \kern-0.5em f(\abs\zvec)\,\prod_{j=1}^n \frac{a_j^{z_j}}{1+a_j}.
 \label{middle}
\end{align}		     
We now apply Lemma~\ref{parity} to estimate the sum on the right hand side.
The small order moments of $X$ are
\begin{equation}
\label{moments}
\begin{split}
\E(X^2) &= \bar{X}^2 + \sum_{j=1}^n p_j(1-p_j),  \\[-1ex]
\E(X^3) &= \bar{X} + 3 \bar{X}^2
+ \bar{X}^3 - 3\sum_{j=1}^n p_j^2 - 3\bar{X}\,\sum_{j=1}^n p_j^2 
           + 2\sum_{j=1}^n p_j^3,\\
\E(X^4)  
   &=  \bar{X} + 7 \bar{X}^2 + 6 \bar{X}^3 + \bar{X}^4
   - \( 6 \bar{X}^2 - 18\bar{X} + 7)\sum_{j=1}^n p_j^2 \\
  & {} \qquad\qquad
    + 3\biggl(\,\sum_{j=1}^n p_j^2\biggr)^2 + (8\bar{X} + 12)\sum_{j=1}^n p_j^3
    - 6\sum_{j=1}^n p_j^4.
\end{split}
\end{equation}
Substituting~\eqref{asums} into these expressions gives
\[ \E(X^k) = 
	     \begin{cases}
        \,1 & \text{~if $k=0$};\\
         \,\sqrt S - S_2/S - \tfrac12 + O(\dmax^3/S^{1/2})
             & \text{~if $k=1$};\\
         \,S - 2S_2/S^{1/2} + O(\dmax^3) & \text{~if $k=2$};\\
         \,S^{3/2} - 3S_2 + \tfrac{3}{2}S + O(\dmax^3S^{1/2})
             & \text{~if $k=3$};\\
         \,S^2 + 4S^{3/2} - 4S^{1/2}S_2 + O(\dmax^3S)
             & \text{~if $k=4$}.
   \end{cases}
\]
Hence
\[
 \hat{f}(p_1,\ldots, p_n) 
 = \E\(f(X)\)\\
       = \exp\biggl(\frac{1}{3S^{1/2}} + \frac{S_2}{2S^{3/2}} +  
                      O\biggl(\frac{\dmax^3}{S}\biggr)\biggr)
\]
where $\hat{f}$ is the function obtained from $f$ as in (\ref{fhatEY}).
Let $f_k$ be the polynomial defined by
$f_k(t) = t^k$ for all $t\in\Reals$, for $k=1,2,3,4$.
Since $\E(X^k) = \hat{f}_k(p_1,\ldots, p_n)$,  replacing each
$p_j$ with $r_j$ in (\ref{moments}) leads to the following
(abusing notation slightly to define the
abbreviation $\hat{f}_1$ in the first line):
\begin{align*}
\hat{f}_1 = \hat{f}_1(r_1,\ldots, r_n) &= \sum_{j=1}^n r_j = O(\sqrt{S}\,),\\[-0.7ex]
\hat{f}_2(r_1,\ldots, r_n) &= \hat{f}_1^2
                                     + \sum_{j=1}^n r_j(1-r_j) = O(S),\\[-0.7ex]
\hat{f}_3(r_1,\ldots, r_n) &= \hat{f}_1 +
   3 \hat{f}_1^2 + \hat{f}_1^3 
  - 3\sum_{j=1}^n r_j^2  - 3\hat{f}_1\,\sum_{j=1}^n r_j^2 
           + 2\sum_{j=1}^n r_j^3 = O(S^{3/2}),\\
\hat{f}_4(r_1,\ldots, r_n)  
   &=  \hat{f}_1 + 7 \hat{f}_1^2 
   + 6 \hat{f}_1^3 + \hat{f}_1^4 - \( 6 \hat{f}_1^2 - 
     18\hat{f}_1 + 7)\sum_{j=1}^n r_j^2  \\
  & {} \qquad\qquad
    + 3\biggl(\,\sum_{j=1}^n r_j^2\biggr)^2 + 
  (8\hat{f}_1 + 12)\sum_{j=1}^n r_j^3
    - 6\sum_{j=1}^n r_j^4 = O(S^4). 
\end{align*}
{}From this we conclude that $\hat f(r_1,\ldots,r_n)= O(1)$. 
Furthermore, 
\[ Z = \exp\biggl(-\Omega\biggl(\sum_{j=1}^n p_j\biggr)\biggr) 
        = e^{-\Omega(\sqrt S)}.  \]
Thus by Lemma~\ref{parity}  we obtain
\[
      \sum_{\abs\zvec\equiv \rho\text{~\rm (mod 2)}}
      \kern-0.5em f(\abs\zvec)\,\prod_{j=1}^n \frac{a_j^{z_j}}{1+a_j}
  = \dfrac{1}{2} \, \exp\biggl(\frac{1}{3S^{1/2}} + \frac{S_2}{2S^{3/2}}
             + O\biggl(\frac{\dmax^3}{S}\biggr)\biggr).\]
Combining this with (\ref{middle}) establishes the lemma when $D=1$.

\bigskip

Next suppose that $D=2$.
Expanding $K^*$ around $\bar{X}$ gives
\begin{equation}
\label{hdef}
 K^*(X) = h(X-\bar{X}) + O(\dmax^3/S)
\end{equation}
where
$h:\Reals\rightarrow\Reals$ is a function
which satisfies
\begin{equation}
\label{easy-h}
 h(y) = O\biggl(\frac{\dmax}{S}\biggr) y + O\biggl(\frac{1}{S}\biggr) y^2
\end{equation}
for $\abs y\le S$.
Recall our assumption that $d_j\geq 1$ for $j=1,\ldots, n$,
which implies that $S\geq n$.
Hence the function $h$ satisfies the conditions of Lemma~\ref{technical}
for some constant $C>0$.   We proceed to apply this lemma, specifically
(\ref{first}).

The second and fourth central moments of $X$ are
\begin{equation}
\label{centrals}
\begin{split}
\E\((X - \bar{X})^2\) &= \sum_{j=1}^n p_j(1-p_j) = O(\bar{X}), \\[-1ex]
\E\((X - \bar{X})^4\) &= 3\, \E\( (X - \bar{X})^2\)^2 +
            \sum_{j=1}^n p_j(1-p_j)(1-6p_j+6p_j^2) = O(\bar{X}+\bar{X}^2). 
\end{split}
\end{equation}
Recall from Lemma~\ref{useful} that $\bar{X} = O(\dmax)$, and
also note that $\abs y \le 1+y^2$.  From \eqref{easy-h} and
\eqref{centrals}, we have
\[  
  \E\( h(X - \bar{X}) \) = O(\dmax/S)
     \( 1 + \E((X-\bar X)^2)\)
  = O(\dmax^2/S).
\]
Similarly, from (\ref{easy-h}) by applying (\ref{centrals})
and using the inequality $(u+v)^2\leq 2(u^2+v^2)$,
we obtain
\[
  \E\( h(X - \bar{X})^2 \) 
    = 
    O(\dmax^2/S^2)\, \E\( (X-\bar{X})^2 \)
 + O(S^{-2})\, \E\( (X - \bar{X})^4 \) \\
  = O(\dmax^3/S^2).
\]
Therefore (\ref{first}) gives
\[
 \E\(\exp(h(X - \bar{X}))\) = \exp\(O(\dmax^3/S)\).
\]
This completes the proof when $D=2$, using (\ref{hdef}).
\end{proof}

We may now prove our main result in the sparse case.

\begin{proof}[Proof of Theorem~\ref{sparsetheorem}]
First suppose that $S > n\log n$.  
Then Lemma~\ref{small-ell} applies for all values of $\ell$.  
Furthermore, 
\[ \dmax^3 = o(S - Dn)\]
since $S - Dn = \Omega(S)$, so (\ref{MWexpression}) can be
applied to $\dvec -D\zvec$, for all 
$\zvec\in\Lambdait$.
Notice also that $a_j=0$ whenever $d_j< D$, so the sum of the right hand
side of (\ref{answer}) over $\zvec\in\Lambdait$ is
equal to the sum over $\{ 0,1\}^n$ when $D=2$, or over
$\Lambdait^{(2)}$ when $D=1$.
Hence the result follows from (\ref{strategy}) using
(\ref{MWexpression}) and Lemmas~\ref{small-ell}--\ref{Kexpectation}.

Now suppose that $n \leq S\leq n\log n$. 
We show that terms with
$\card{\zvec} > S/3$ give a negligible contribution
to $G_D(\dvec)$.

It is well known that when $S$ is even, we can write
\[ G(\dvec) = \frac{S!}{(S/2)!\, 2^{S/2}}
  \biggl(\,\prod_{j=1}^n d_j!\biggr)^{\!-1} P(\dvec)
\]
where $P(\dvec)$ is a probability, and hence is at most 1.
(Indeed, the $\exp(\cdot)$ factor in (\ref{MWexpression}) is an 
approximation to $P(\dvec)$ when $\dmax = o(S^{1/3})$,
as proved in~\cite{MWsparse}.)
It follows by Stirling's approximation that
\[
G(\dvec) = O(1)\,
 \biggl(\frac{S}{e}\biggr)^{S/2}
  \biggl(\,\prod_{j=1}^n d_j!\biggr)^{\!-1}\,
  \]
for any even value of $S$.
Recall the definition of $\Lambdait_\ell$ from (\ref{lambda-ell-def}).
For $\zvec\in\Lambdait_\ell$ we have
\[ G(\dvec - D\zvec) =  O(1)\,
         \biggl(\frac{S-D\ell}{e}\biggr)^{(S-D\ell)/2}\,
           \biggl(\,\prod_{j=1}^n \, (d_j - Dz_j)!\biggr)^{\!-1}.\]
Furthermore, 
\[ H(\dvec)^{-1} = \exp\(O( \dmax^2)\)\, 
              \biggl(\frac{e}{S}\biggr)^{\!S/2}
   \, \prod_{j=1}^n d_j!.
\]
Hence 
\[
 \frac{G(\dvec - D\zvec)}{H(\dvec)}
        = O(1)\exp\(O(\dmax^2)\)\,
   \biggl(\frac{\dmax^2\, e}{S}\biggr)^{\!D\ell/2}. 
\]
Therefore, recalling that $\ell\leq n$ and ignoring parity for
an upper bound,
\begin{align}
\sum_{\ell = S/3 }^n \,\,
  \sum_{\zvec\in \Lambdait_\ell} 
   \frac{G(\dvec - D\zvec)}{H(\dvec)}
  &= O(1) \exp\(O(\dmax^2)\)\, \sum_{\ell = S/3}^n \, \binom{n}{\ell} 
          \biggl(\frac{\dmax^2 e}{S}\biggr)^{D\ell/2} \notag \\
  &= O(1)\exp\(O(\dmax^2)\)\, \sum_{\ell = S/3}^n \, \binom{n}{\ell}
    S^{-D\ell/6} \notag \\
  &= O(1)\exp\(O(\dmax^2)\)\, 2^n\, S^{-DS/18} \notag \\[0.5ex]
  &= O\(S^{-\Omega(S)}\).\label{earlier}
\end{align}
Recall that (\ref{MWexpression}) applies when $\ell < S/3$.
Therefore, using (\ref{MWexpression}) and Lemma~\ref{small-ell},
\begin{align*}
 \frac{G_D(\dvec)}{H(\dvec)} &= 
  S^{- \Omega(S)} + 
           \sum_{\ell=0}^{S/3}\,\,  \sum_{\zvec\in \Lambdait_\ell}
            \frac{G(\dvec - D\zvec)}{H(\dvec)}
   \notag \\
  &= O\(S^{- \Omega(S)}\) + 
       \exp\(O(\dmax^3/S)\)\,
     \sum_{\ell = 0}^{S/3}\, \sum_{\zvec\in\Lambdait_\ell}
  \exp\(K(\zvec)\)\, \prod_{j=1}^n a_j^{z_j}. 
\end{align*}
Hence, by \eqref{Kbounds},
\begin{align}
  O\(S^{- \Omega(S)}\) + &
       \exp\(O(\dmax^3/S)\)\,
     \sum_{\ell = 0}^{S/3}\, \exp\(K'(\ell)\) \,
    \sum_{\zvec\in\Lambdait_\ell}\, \prod_{j=1}^n a_j^{z_j}\notag\\
  & \leq  \frac{G_D(\dvec)}{H(\dvec)} \notag\\
  & \leq 
  O\(S^{- \Omega(S)}\) + 
       \exp\(O(\dmax^3/S)\)\,
     \sum_{\ell = 0}^{S/3}\, \exp\(K''(\ell)\) \,
    \sum_{\zvec\in\Lambdait_\ell}\, \prod_{j=1}^n a_j^{z_j}. 
\label{end}
\end{align}
Next we would like to show that, in either the lower or upper bound
in (\ref{end}), the sum over $\ell$ can be extended up to
$\ell=n$ without affecting the answer significantly.
Since every term is positive, zero is a lower bound for the tail
of the sum.
Again, we ignore the parity issue for an upper bound.
Let $K^\ast$ be either $K'$ or $K''$.
Firstly, note that since $n\leq S$ we have $K^*(\ell) = O(\ell)$
uniformly for $S/3 \leq \ell\leq n$.
Furthermore $a_j = o(S^{-D/6})$ for $j=1,\ldots, n$.
Therefore
\begin{align}
\sum_{\ell= S/3}^n \, \exp\( K^\ast(\ell)\) \, \sum_{\zvec\in\Lambdait_\ell}
    \, \prod_{j=1}^n a_j^{z_j}
	   &\leq
\sum_{\ell= S/3}^n \, \binom{n}{\ell}\, \(e^{O(1)}\, S^{-D/6}\)^\ell \notag\\
&\leq
\sum_{\ell= S/3}^n \, \binom{n}{\ell}\, S^{-D\ell/7} \notag\\
&= O\(S^{-\Omega(S)}\) \label{pipsqueek}
\end{align}
as in (\ref{earlier}).
Combining this with (\ref{end}) gives 
\begin{align*}
  O\(S^{- \Omega(S)}\) + &
       \exp\(O(\dmax^3/S)\)\,
    \sum_{\zvec\in\Lambdait}\, \exp\( K'(\abs\zvec)\)\,
          \prod_{j=1}^n a_j^{z_j}\notag\\
  & \leq  \frac{G_D(\dvec)}{H(\dvec)} \notag\\
  & \leq 
  O\(S^{- \Omega(S)}\) + 
       \exp\(O(\dmax^3/S)\)\,
    \sum_{\zvec\in\Lambdait}\, 
      \exp\(K''(\abs\zvec)\) \, \prod_{j=1}^n a_j^{z_j}. 
\label{end}
\end{align*}
The result now follows from Lemma~\ref{useful} and Lemma~\ref{Kexpectation}.
\end{proof}

\nicebreak
\section{Proof of Theorem~\ref{distribution}}\label{s:distributionproof}

Part (i).~
Under the conditions of Theorem~\ref{densetheorem}, the distribution
of $Y_D$ follows directly from \eqref{binexp} and \eqref{bintail}, noting
in the case of $D=1$ that the restriction of $\ell$ to the same parity
as $S$ changes the normalizing factor by 2 to high precision, as
explained in the last paragraph of Section~\ref{ss:dense-proof}.
The formula for the expectation follows on summing $\ell\Prob(Y_D=\ell)$, since
the error term $O(e^{-n^{\Omega(1)}})$ contributes negligibly.
To see that the same is true for the variance, it helps to use the
cancellation-free formula
\begin{equation}
\label{varsum}
 \Var(Z) = \sum_{k<\ell} \Prob(Z=k)\Prob(Z=\ell)\, (k-\ell)^2,
\end{equation}
which is true for all discrete random variables $Z$ of finite variance
(see for example~\cite[p.~8]{lee}).

Part (ii).~
Now suppose that the conditions of Theorem~\ref{sparsetheorem}
hold and consider the case $D=1$.
Let $X$ be a random variable with the Poisson binomial 
distribution $\PB(\pvec)$, where $p_j=a_j/(1+a_j)$ and $a_j$ is
defined in Lemma~\ref{small-ell}.
{}From \eqref{answer}, \eqref{Kexpand}, we find that for 
$\ell=\sqrt{S}+O(S^{1/3})$, the distribution of $Y_1$ is proportional
to $\PB(\pvec)$ to relative error $O(\dmax^3/S+S^{-1/3})$.
Moreover, the weight of both $Y_1$ and $X$ from
$\abs{\ell-\sqrt{S}\,} > S^{1/3}$ is $e^{-S^{\Omega(1)}}$, and
restriction of $\ell$ to the same parity as $S$ contributes a factor
of 2 to high precision as in the proof of Lemma~\ref{Kexpectation}.
This gives, for $\ell=0,\ldots,n$,
\[ \Prob(Y_1=\ell) 
     = \(2 + O(\dmax^3/S +S^{-1/3})\)
                    \PB(\pvec,\ell) +O(e^{-S^{\Omega(1)}}).
\]
Next we show that the parameters $\pvec'$ in the theorem are
sufficiently close to the parameters~$\pvec$.  For each $j$, we
find that
\begin{equation}\label{pjpj}
   p_j = \exp\(O(\dmax^3/S^{3/2} +S^{-1})\) \, p'_j.
\end{equation}
By definition,
\[
   \PB(\pvec,\ell) = \sum_{\abs W=\ell} \biggl(\,
       \prod_{j\in W} p_j \prod_{j\notin W} (1-p_j) \biggr),
\]
where the sum is over subsets $W\subseteq\{1,2,\ldots,n\}$ of size $\ell$.
Applying~\eqref{pjpj}, we find that
\[
  \PB(\pvec,\ell) = \exp\(O(\dmax^3/S +S^{-1/2})\)\, \PB(\pvec',\ell)
\]
for $\ell=O(\sqrt{S}\,)$. The tail past $\sqrt{S}$ is 
$e^{-S^{\Omega(1)}}$ for both $\PB(\pvec)$ and $\PB(\pvec')$, by
Lemma~\ref{Chernoff}.
This completes the proof of the distribution
for $D=1$.

The mean and variance follow as for part~(i) to the same relative
precision
as the distribution, but we can do better by using the more accurate
distribution analysed in the proof of Lemma~\ref{Kexpectation}.
As we have shown in~\eqref{Kexpand},
for $\ell=\sqrt S + O(S^{1/3})$, which excludes only
exponentially small tails, 
\begin{equation}
\label{Xcentral}
  \Prob(Y_1=\ell) \propto \exp\(O(\dmax^3/S)\) \Prob(X=\ell) f(\ell)
\end{equation}
if $\ell$ has the same parity as~$S$.
Define the discrete random variable $Z$ by
\[
   \Prob(Z=t) \propto \Prob(X=\bar X + t) f(\bar X + t),
\]
whenever $\bar X + t$ is an integer in $[0,n]$ with the same parity
as $S$; and $\Prob(Z=t)=0$ otherwise.
By~\eqref{Xcentral} and the argument used in part (i) of this proof,
\begin{align*}
   \E(Y_1) &= \exp\(O(\dmax^3/S)\) \( \bar X + \E(Z)\)  ,\\
   \Var(Y_1) &= \exp\(O(\dmax^3/S)\) \Var(Z) .
\end{align*}

For $m\ge 0$, define the central moment $\mu_m=\E\( (X-\bar X)^m \)$
and the cumulant $\kappa_m$ by
\[
   \log\phi(t) = \sum_{m=1}^\infty \kappa_m (it)^m/m!\, ,
\]
where $\phi(t)=\prod_{j=1}^n \(p_je^{it}+1-p_j\)$ is the
characteristic function of~$X$.
We find that $\kappa_m=O(\sqrt{S}\,)$ for $2\le m\le 6$. 
Using the well-known expressions for the central moments in terms
of the cumulants, and 
the explicit formulae~\eqref{centrals}, we find that
\begin{align*}
  \mu_2 &= \sqrt S - \frac{2S_2}{S} - \frac{3}{2} + O(\dmax^3/S^{1/2}), \\
  \mu_3 &= \kappa_3 = O(\sqrt{S}\,), \\
  \mu_4 &= 3S + O(\dmax S^{1/2}), \\
  \mu_5 &= \kappa_5 + 10 \kappa_3 \kappa_2 = O(S), \\
  \mu_6 &= \kappa_6 + 15 \kappa_4 \kappa_2 + 10 \kappa_3^2
                + 15 \kappa_2^3 = O(S^{3/2}).
\end{align*}
Thus we calculate
\begin{align*}
  M_0 &= \sum_{\ell=0}^n\sumprime \Prob(X=\ell)\, f(\ell) 
        = \frac{1}2{} + \frac{1}{6S^{1/2}} + \frac{S_2}{4S^{3/2}} + O(\dmax^2/S), \\
  M_1 &= \sum_{\ell=0}^n\sumprime \Prob(X=\ell)\, f(\ell) \,(\ell - \bar X) 
        = O(\dmax^2/S+S^{-1/2}), \\
  M_2 &= \sum_{\ell=0}^n\sumprime \Prob(X=\ell)\, f(\ell)\, (\ell - \bar X)^2
        = \frac{\sqrt S}{2} - \frac{3S_2}{4S}  - \frac{1}{3} + O(\dmax^3/S^{1/2}),
\end{align*}
where the primes indicate that the sums are restricted to $\ell$
having the same parity as~$S$.
The effect of the parity restriction is handled in the same way
as in the proof of Lemma~\ref{Kexpectation}, and in fact the first
summation is equivalent to Lemma~\ref{Kexpectation}.
Now we have that $\E(Z) = M_1/M_0$ and 
$\Var(Z) = M_2/M_0 - \E(Z)^2$.
{}From these the mean and variance of $Y_1$ follow.

Finally we consider part (ii) in the case $D=2$.  Define $X$ as
before, with $a_j$ as in Lemma~\ref{small-ell}.
By Lemma~\ref{Chernoff}, $\Prob(X\ge S^{1/3}) = O(e^{-S^{\Omega(1)}})$.
The same bound holds for $\Pr(Y_2 > S/3)$, using the argument leading
to~\eqref{earlier}.  Combining this with (\ref{contribution}) and 
Lemma~\ref{small-ell} shows
that $\Pr(Y_2\geq 2S^{1/2}) = O(e^{-S^{\Omega(1)}})$.
Finally, since $K''(\ell) = O(1)$ for $\ell = O(S^{1/2})$,
we conclude that $\Pr(Y_2 \geq S^{1/3} ) = O(e^{-S^{\Omega(1)}})$.

Lemma~\ref{small-ell} shows that for $\ell\le S^{1/3}$,
\[ 
   \Prob(Y_2=\ell) =\exp\( O(\dmax^3/S+S^{-1/3}) \)\, \Prob(X=\ell).
\]
By the argument above, the ratio of $\PB(\pvec)$ to
$\PB(\pvec'')$ for $\ell\le S^{1/3}$ is
$\exp\( O(\dmax^2/S^{2/3})\)$, since
$p_j = \exp\(O(\dmax^2/S)\) p_j''$ for all~$j$.
The given estimate of the distribution of $Y_2$ follows.

To obtain the mean and variance of $Y_2$, we use the sharper estimate
\[ \Prob(Y_2=\ell) \propto \exp\( O(\dmax^3/S)\) \Prob(X=\ell)\, (1 + \ell^2/S), \]
valid for $\ell\le \dmax^{3/4} S^{1/4}$ by Lemma~\ref{small-ell}, with the
weight of the tail $\ell > \dmax^{3/4} S^{1/4}$ being
exponentially small as usual.  Using~\eqref{moments} we find that
$\E(X^2)=O(\dmax^2)$, $\E(X^3)=O(\dmax^2 S_2/S)$ and
$\E(X^4)=O(\dmax^3 S_2/S)$, and so
\begin{align*}
  \sum_{\ell=0}^n \Prob(X=\ell)\,(1 + \ell^2/S)
        &= 1 + O\biggl(\frac{\dmax^2}{S}\biggr), \\
  \sum_{\ell=0}^n \Prob(X=\ell)\,(1 + \ell^2/S) \,\ell
        &= \frac{S_2}{S} + O\biggl(\frac{\dmax^2S_2}{S^2}\biggr), \\
   \sum_{\ell=0}^n \Prob(X=\ell)\,(1 + \ell^2/S) \,\ell^2
        &= \frac{S_2^2}{S^2} + \frac{S_2}{S}
		+ O\biggl(\frac{\dmax^3S_2}{S^2}\biggr),
\end{align*}
and from these the expressions for $\E(Y_2)$ and $\Var(Y_2)$ follow,
recalling the cancellation-free variance formula~\eqref{varsum}.
\qed

\section{A conjecture for regular graphs with loops}\label{s:conjecture}

In the case of $D=2$ and $\dvec=(d,d,\ldots,d)$, an informal
computation provides motivation for the sparse and dense enumeration
formulae and suggests a more general conjecture.
Since $D=2$ we have $d\in \{ 0,1,\ldots, n+1\}$.
Recall the notations $G_2(n,d) = G_2(d,d,\ldots,d)$ and $\mu_2=d/(n+1)$. 

Generate a random $n$-vertex graph by independently choosing each of the
$\binom{n+1}{2}$ possible edges (including loops) with probability~$\mu_2$.
Each $d$-regular graph has exactly $nd/2$ edges, so it
occurs with probability
\begin{equation}\label{peach}
  \mu_2^{nd/2} \, (1-\mu_2)^{\binom{n+1}{2}-nd/2}.
\end{equation}
The event that a particular vertex has degree $d$  has probability
\begin{align}\label{prow}
  \binom{n-1}{d}\, \mu_2^d\, (1-\mu_2)^{n-d} &+ 
   \binom{n-1}{d-2}\, \mu_2^{d-1}\, (1-\mu_2)^{n-d+1} \notag\\[1ex]
   &= \binom{n+1}{d} \, \frac{n-1}{n}\,   \mu_2^{d} \, (1-\mu_2)^{n-d+1}.
\end{align}
If the vertex degrees were independent (which of course they are not), the number
of graphs would be the $n$-th power of \eqref{prow} divided by \eqref{peach}.
Noting that $(1-1/n)^n\to e^{-1}$, this gives a ``na\"\i ve'' estimate
\[
  \widehat G_2(n,d) = e^{-1}
  \binom{n+1}{d}^{\!n} 
  \( \mu_2^{\mu_2} (1-\mu_2)^{1-\mu_2}\)^{\binom{n+1}{2}}.
\]
We can see from Theorem~\ref{densetheorem} that $G_2(n,d)$ is
larger than  $\widehat G_2(n,d)$ by a factor close to $\sqrt 2\, e^{1/4}$
whenever $\min\{d,n-d\}>cn/\log n$ for some constant $c>\frac23$.
Less obviously, the same is true for $1\le d = o(n^{1/2})$ by
Theorem~\ref{sparsetheorem}.
Recall that the same constant $\sqrt2\,e^{1/4}$ appears in a similar
context for regular graphs without loops~\cite{MWreg}.
This leads us to investigate the region between the coverage of our
sparse and dense theorems.

Using the method described in \cite{Mlabelled}, we computed the exact
values of $G_2(n,d)$ for about 150 nontrivial values of $(n,d)$ up to $n=35$.
For example,
\[ 
G_2(22,10) = 7789744323722189254716829156528211234980743220762340514888.
\]
Numerical analysis of these values suggests the following analogue of
\cite[Conj.~2]{MWreg}.
\begin{conj}
  Let $d=d(n)$ satisfy $1\le d\le n$ with $nd$ even.
  Then
  \[ 
    G_2(n,d) = \sqrt2\, \binom{n+1}{d}^{\!n} 
    \( \mu_2^{\mu_2} (1-\mu_2)^{1-\mu_2}\)^{\binom{n+1}{2}}
    \exp\biggl( -\frac34 + \frac{3c+1}{12cn}+ O(n^{-2}) \biggr)
  \]
  uniformly as $n\to\infty$, where $\mu_2=d/(n+1)$ and
  $c = \mu_2(1-\mu_2)(n+1)$.
\end{conj}
The numerical evidence suggests
that in fact the term $O(n^{-2})$ always lies in the
interval $(-2/n^2,0)$ for $n\ge 4$.

\nicebreak

\end{document}